\PassOptionsToPackage{unicode}{hyperref}
\PassOptionsToPackage{hyphens}{url}
\PassOptionsToPackage{dvipsnames,svgnames,x11names}{xcolor}
\documentclass[12pt, letterpaper]{article}

\usepackage{amsmath,amssymb}
\usepackage{iftex}
\ifPDFTeX
  \usepackage[T1]{fontenc}
  \usepackage[utf8]{inputenc}
  \usepackage{textcomp} 
\else 
  \usepackage{unicode-math}
  \defaultfontfeatures{Scale=MatchLowercase}
  \defaultfontfeatures[\rmfamily]{Ligatures=TeX,Scale=1}
\fi
\usepackage{lmodern}
\ifPDFTeX\else
\fi
\IfFileExists{upquote.sty}{\usepackage{upquote}}{}
\IfFileExists{microtype.sty}{
  \usepackage[]{microtype}
  \UseMicrotypeSet[protrusion]{basicmath} 
}{}
\makeatletter
\@ifundefined{KOMAClassName}{
  \IfFileExists{parskip.sty}{%
    \usepackage{parskip}
  }{
    \setlength{\parindent}{0pt}
    \setlength{\parskip}{6pt plus 2pt minus 1pt}}
}{
  \KOMAoptions{parskip=half}}
\makeatother
\usepackage{xcolor}
\makeatletter
\ifx\paragraph\undefined\else
  \let\oldparagraph\paragraph
  \renewcommand{\paragraph}{
    \@ifstar
      \xxxParagraphStar
      \xxxParagraphNoStar
  }
  \newcommand{\xxxParagraphStar}[1]{\oldparagraph*{#1}\mbox{}}
  \newcommand{\xxxParagraphNoStar}[1]{\oldparagraph{#1}\mbox{}}
\fi
\ifx\subparagraph\undefined\else
  \let\oldsubparagraph\subparagraph
  \renewcommand{\subparagraph}{
    \@ifstar
      \xxxSubParagraphStar
      \xxxSubParagraphNoStar
  }
  \newcommand{\xxxSubParagraphStar}[1]{\oldsubparagraph*{#1}\mbox{}}
  \newcommand{\xxxSubParagraphNoStar}[1]{\oldsubparagraph{#1}\mbox{}}
\fi
\makeatother

\usepackage{longtable,booktabs,array}
\usepackage{calc} 
\usepackage{etoolbox}
\makeatletter
\patchcmd\longtable{\par}{\if@noskipsec\mbox{}\fi\par}{}{}
\makeatother
\IfFileExists{footnotehyper.sty}{\usepackage{footnotehyper}}{\usepackage{footnote}}
\makesavenoteenv{longtable}
\usepackage{graphicx}
\makeatletter
\def\maxwidth{\ifdim\Gin@nat@width>\linewidth\linewidth\else\Gin@nat@width\fi}
\def\maxheight{\ifdim\Gin@nat@height>\textheight\textheight\else\Gin@nat@height\fi}
\makeatother
\setkeys{Gin}{width=\maxwidth,height=\maxheight,keepaspectratio}
\makeatletter
\def\fps@figure{htbp}
\makeatother

\makeatletter
\@ifpackageloaded{caption}{}{\usepackage{caption}}
\AtBeginDocument{%
\ifdefined\contentsname
  \renewcommand*\contentsname{Table of contents}
\else
  \newcommand\contentsname{Table of contents}
\fi
\ifdefined\listfigurename
  \renewcommand*\listfigurename{List of Figures}
\else
  \newcommand\listfigurename{List of Figures}
\fi
\ifdefined\listtablename
  \renewcommand*\listtablename{List of Tables}
\else
  \newcommand\listtablename{List of Tables}
\fi
\ifdefined\figurename
  \renewcommand*\figurename{Figure}
\else
  \newcommand\figurename{Figure}
\fi
\ifdefined\tablename
  \renewcommand*\tablename{Table}
\else
  \newcommand\tablename{Table}
\fi
}
\@ifpackageloaded{float}{}{\usepackage{float}}
\floatstyle{ruled}
\@ifundefined{c@chapter}{\newfloat{codelisting}{h}{lop}}{\newfloat{codelisting}{h}{lop}[chapter]}
\floatname{codelisting}{Listing}

\makeatother
\makeatletter
\@ifpackageloaded{caption}{}{\usepackage{caption}}
\@ifpackageloaded{subcaption}{}{\usepackage{subcaption}}
\makeatother

\ifLuaTeX
  \usepackage{selnolig}  
\fi
\usepackage[]{natbib}
\usepackage{bookmark}

\IfFileExists{xurl.sty}{\usepackage{xurl}}{} 
\hypersetup{
  pdftitle={The Categorical Instrumental Variable Model:
  Characterization, Partial Identification, and Statistical Inference},
  pdfauthor={Yilin Song, F. Richard Guo, K.C. Gary Chan, Thomas S. Richardson},
  pdfkeywords={instrumental variable, partial identification, Strassen's theorem,
  average treatment effect, confidence region, concentration inequality},
  colorlinks=true,
  linkcolor={blue},
  filecolor={Maroon},
  citecolor={Blue},
  urlcolor={Blue},
  pdfcreator={LaTeX via pandoc}
}

\newcommand{\anon}{1} 

\usepackage{amsthm}
\usepackage{authblk}
\usepackage{mathtools}
\usepackage[capitalise]{cleveref}
\usepackage{bm}
\usepackage[shortlabels]{enumitem}
\usepackage{multirow}
\usepackage{rotating}
\usepackage{algorithm}
\usepackage{algorithmic}
\usepackage{bbding}

\usepackage{tikz}
\usetikzlibrary{automata, arrows,shapes.arrows,shapes.geometric,
shapes.multipart,backgrounds,decorations.text,decorations.pathmorphing,decorations.pathreplacing,positioning,fit,swigs}
\usetikzlibrary{trees}
\usetikzlibrary{calc}
\usetikzlibrary{shapes,decorations,arrows,arrows.meta,fit,positioning}
\usetikzlibrary{shapes.multipart}
\tikzset{
    -Latex,auto,node distance =1 cm and 1 cm,semithick,
    state/.style ={ellipse, draw, minimum width = 0.7 cm},
    state1/.style ={ draw, minimum width = 0.7 cm},
    point/.style = {circle, draw, inner sep=0.04cm,fill,node contents={}},
    bidirected/.style={Latex-Latex,dashed},
    el/.style = {inner sep=2pt, align=left, sloped}
}

\newcommand{\ind}{\mathrel{\text{\scalebox{1.07}{$\perp\mkern-10mu\perp$}}}}
\newcommand{\Nc}{{\cal N}_{\mathcal{R}_C}}
\newcommand{\Rc}{\mathcal{R}_{C}}
\newcommand{\calv}{{\cal V}}

\newcommand{\calm}{{\cal M}}
\newcommand{\bmv}{{\bm v}}

\newcommand{\KL}{\mathcal{D}_{\text{KL}}}
\newcommand{\bigO}{\mathcal{O}}

\newcommand{\calM}{\mathcal{M}}

\newcommand{\spacingset}[1]{%
  \renewcommand{\baselinestretch}{#1}%
  \small
  \normalsize
}

\newcommand{\setsupplementnumbering}{%
  \setcounter{equation}{0}%
  \setcounter{table}{0}%
  \setcounter{theorem}{0}%
  \setcounter{lemma}{0}%
  \setcounter{proposition}{0}%
  \setcounter{corollary}{0}%
  \setcounter{section}{0}%
  \setcounter{figure}{0}%
  \setcounter{example}{0}%
  \setcounter{assumption}{0}%
  \renewcommand{\theequation}{S\arabic{equation}}%
  \renewcommand{\thetable}{S\arabic{table}}%
  \renewcommand{\thefigure}{S\arabic{figure}}%
  \renewcommand{\thesection}{S\arabic{section}}%
  \renewcommand{\thetheorem}{S\arabic{theorem}}%
  \renewcommand{\thelemma}{S\arabic{lemma}}%
  \renewcommand{\theproposition}{S\arabic{proposition}}%
  \renewcommand{\thecorollary}{S\arabic{corollary}}%
  \renewcommand{\theexample}{S\arabic{example}}%
  \renewcommand{\theassumption}{S\arabic{assumption}}%
}

\newcommand{\red}{\color{red}}
\newcommand{\blue}{\color{blue}}

\usepackage{xr}
\theoremstyle{plain}
\newtheorem{theorem}{Theorem}
\newtheorem*{theorem*}{Theorem}
\newtheorem{corollary}{Corollary}

\newtheorem{proposition}{Proposition}
\newtheorem{remark}{Remark}
\newtheorem{lemma}{Lemma}
\newtheorem{definition}{Definition}
\newtheorem{assumption}{Assumption}
\crefname{assumption}{Assumption}{Assumptions}

\theoremstyle{plain}

\theoremstyle{remark}

\newtheorem{example}{Example}

\begin{document}

\def\spacingset#1{\renewcommand{\baselinestretch}%
{#1}\small\normalsize} \spacingset{1}

\if1\anon
{
  \title{\bf The Categorical Instrumental Variable Model:\\ Characterization, Partial Identification, and Statistical Inference}
  \author{Yilin Song$^1$\thanks{Correspond to ys4057@cumc.columbia.edu.}, F.~Richard Guo$^2$, K.C.~Gary Chan$^3$, Thomas~S.~Richardson$^4$\vspace{.4cm}\\ 
    $^1$Department of Biostatistics, Columbia University \\
    $^2$Department of Statistics, University of Michigan \\
    $^3$Department of Biostatistics, University of Washington\\
    $^4$Department of Statistics, University of Washington
    }
    \date{}
  \maketitle
} \fi

\if0\anon
{
  \bigskip
  \bigskip
  \bigskip
  \begin{center}
    {\LARGE\bfseries
  The Categorical Instrumental Variable Model: \\ Characterization, Partial Identification, and Statistical Inference\par}
\end{center}
  \medskip
} \fi






\begin{abstract}

We study categorical instrumental variable (IV) models with instrument, treatment and outcome taking finitely many values. 
We derive a simple closed-form characterization of the set of joint distributions of potential outcomes that are compatible with a given observed data distribution in terms of a minimal set of inequalities.
These inequalities unify several different IV models defined by versions of the independence and exclusion restriction assumptions. They lead to sharp bounds on causal functionals and provide a sharp criterion for model falsification.
For linear functionals of the joint counterfactual distribution, such as pairwise average treatment effects and probabilities of potential outcomes, we construct confidence intervals with simultaneous finite-sample coverage, using a tail bound on the Kullback--Leibler divergence.
We illustrate our method using data from the Minneapolis Domestic Violence Experiment.

\end{abstract}

\noindent%
\vfill

\spacingset{1.8} 


\vspace{-1cm}\section{Introduction}\vspace{-0.3cm}

This article studies partial identification of instrumental variable (IV) models in which the instrument, treatment, and outcome are categorical. 

Let $X$ and $Y$ denote the exposure and outcome of interest respectively. Generally speaking, a variable $Z$ is a valid instrumental variable if certain versions of the following two assumptions hold: (1) an independence (or exchangeability) condition: $Z$ is independent of any unmeasured confounder $U$ of the treatment-outcome relationship; (2) an exclusion restriction: there is no direct effect of $Z$ on the outcome $Y$ other than through the treatment of interest $X$. Both of these assumptions are individually untestable. A third relevance assumption, which states that $Z$ is associated with the treatment $X$, is also often invoked in the IV literature. However, for the purposes of our analysis it is not required; our bounds will still be valid under any association between $Z$ and $X$. 
A directed acyclic graph (DAG) representing the assumptions on instrumental variables is shown in \cref{fig:dag_assump}.

\vspace{-0.4cm} \subsection{Motivating example: Minneapolis domestic violence experiment}\label{sec:motivating} \vspace{-0.3cm}
To illustrate our approach, we consider the Minneapolis domestic violence experiment \citep{berkandsherman}: the Minneapolis Police Department and the Police Foundation conducted an experiment from early 1981 to mid-1982 for testing the relation between police response to domestic violence and whether the suspect subsequently re-offended. 
When the officers responded to a domestic violence case, they were randomly recommended by lottery to take one of three courses of action: arrest the suspect; 
 send the suspect from the scene of the assault for eight hours;
 or provide advice. Following
 \citet{berkandsherman} and \citet{angristcrime}, we will label these strategies {\em Arrest}, {\em Separate} and {\em Advise}.\footnote{In the original experiment, {\em Separate} was denoted {\em Send}.} The study followed up all cases after a 6-month period to determine whether the suspect had re-offended either via self-reports or from a police database. There were a total of 313 cases in the experiment. Through random assignment, 92 cases were recommended to {\em Arrest}, 108 to {\em Advise}, and 113 to {\em Separate}.

In many randomized controlled trials (RCTs), participants may not receive their assigned treatment, leading to non-compliance. In the Minneapolis experiment, a responding police officer had the option to implement a different response ($X$) from the one that they were randomly recommended ($Z$), resulting in non-compliance. The full data are shown in \cref{tab:data}, where $Y=2$ indicates that the suspect re-offended during the 6-month follow-up period and $Y=1$ indicates otherwise. Notice that in many cases $X\neq Z$, suggesting that the officer did not adhere to the recommended action.

\begin{table}[!ht]
\centering
\caption{\centering Minneapolis Domestic Violence Experiment \citep[Table 2]{berk1988police}:\\
each cell shows \#(no re-offence) / \#(re-offence) in 6 months}
\label{tab:data}
\resizebox{0.6\linewidth}{!}{%
\begin{tabular}{c | c c c | c}
\toprule
$Y=1$ / $Y=2$ & $X=\text{Arr}$ & $X=\text{Adv}$ & $X=\text{Sep}$ & Total \\
\midrule
$Z=\text{Arr}$ & \boxed{{81} / {10}} & {0} / {0} & {1} / {0} & {82} / {10} \\
$Z=\text{Adv}$ & {15} / {3} & \boxed{{69} / {15}} & {3} / {3} & {87} / {21} \\
$Z=\text{Sep}$ & {21} / {5} & {4} / {1} & \boxed{{62} / {20}} & {87} / {26}\\
\bottomrule
\end{tabular}
}
\end{table}

\vspace{-0.2cm}Historically the problem of treatment non-compliance was often addressed via either an Intention-to-treat (ITT) or Per Protocol (PP) analysis \citep{PPvsITT}. The ITT approach analyzes the effect of the assigned treatment, regardless of any subsequent non-compliance (e.g., the last column of \cref{tab:data}). Though the ITT causal estimand is identified, due to random assignment, it does not assess the efficacy of the treatment itself, and may lack ecological validity since compliance behavior may be highly context dependent. Meanwhile, the PP analysis considers only those participants who fully adhered to their assigned treatment protocol (e.g., the boxed cells in \cref{tab:data}). However, this comparison will not, in general, be causal because the set of people who follow the protocol in one treatment arm may not be comparable with those who do so in another arm.

In the case of a binary treatment, a third approach, pioneered by 
\cite{imbens-angrist}, focuses on the causal effect of treatment among the {\it compliers}---those who would have taken their assigned treatment no matter which arm they were assigned to---referred to as the local average treatment effect (LATE) \citep{angrist2sls, identifictionIV1996}. An advantage of this framework is that no further work is required, conceptually, to specify the circumstances under which such a subject would have taken treatment or control. 

However, the LATE is not identified without additional assumptions, such as that there are no {\it defiers}, defined as those who would always take the treatment opposite to their assignment. 
Although this assumption has testable implications, and thus may be falsified, it must be argued for on substantive grounds, which may not apply in every setting. Also, though useful in establishing the existence of a stratum (i.e., compliers) in which the treatment does have an effect, it is less clear how this should inform specific decisions to use or withhold treatment, since this requires determining whether a subject would have been a complier had they been in the experiment \citep{kennedy:classifying-compliers:2020}. For example, in the case of the Minneapolis study, this would require judging that had a domestic violence incident happened during the course of the study, the responding officer would have judged it appropriate to use the assigned strategy, whatever that was. \citet{smallthree-arm} consider extensions of the no defier assumption in a setting where there are two active treatment arms and a placebo arm; see also \citet{heckman-vytlacil:2005} for related work.

By design, the LATE does not assess the consequences of adopting a uniform policy to be applied to all subjects, including non-compliers. In contrast, the approach that we consider here focuses on the average treatment effect (ATE) of the treatment on the outcome, which would be identified in an experiment on the same population, and measures the causal effect of the treatment itself on the whole population without regard to the compliance behavior and thus may be more relevant to policy decisions \citep{robins:discussion:1996}. At the same time, consideration of this global ATE does presume that it is meaningful, at least conceptually, to consider applying (not merely assigning) each treatment to every subject.

Without additional assumptions, the ATE is only partially identified \citep{manski1990,robins:1989:ivbounds}, but non-trivial bounds that exclude zero can be obtained even in settings with substantial non-compliance \citep{BPbounds}. 
The case in which the treatment $X$ is binary has been studied extensively. Sharp bounds have also been obtained for a binary treatment when the instrument takes more levels \citep{RichardsonRobins}. For example, in an encouragement design, subjects may be assigned to several different levels of (financial) incentive to start treatment. 

However, approaches that may be applied to studies, such as the Minneapolis experiment, in which the \emph{treatment} itself takes more than two levels are less well developed. 
This presents a challenge even for a researcher who is primarily interested in the ATE contrasting only two treatments, such as {\em Advise} vs.~{\em Separate}.\footnote{These were collectively referred to as ``Coddling'' strategies by \citet{angristcrime}.} Since $X$, the treatment received, was not randomized, the availability of the third treatment ({\em Arrest}) must be accounted for. In particular, it would be inappropriate to apply bounds developed for binary $X$, by simply deleting the first column from \cref{tab:data}. Such an approach implicitly conditions on $X\neq$ {\em Arrest}. This is problematic because $X$ is typically affected by the random assignment $Z$; it is quite possible that there are individuals who would not be arrested if assigned to {\em Separate}, but would be arrested if assigned to {\em Advise}. Consequently, the sets of subjects who were {\em not} arrested in different $Z$ arms are not necessarily comparable (in other words, conditioning on $X\neq$ {\em Arrest} can break the independence between $Z$ and $U$ in \cref{fig:dag_assump}), and thus were we to calculate the IV bounds for binary $X$, using the counts in the {\em Advise} and {\em Separate} columns of \cref{tab:data}, the resulting bounds would not be guaranteed to cover the ATE comparing {\em Advise} versus {\em Separate} \citep{Swanson:biasselection}.

\vspace{-0.5cm}\subsection{Contribution of the paper}\vspace{-0.3cm}

Our paper addresses this methodological gap: we provide a simple characterization, via linear inequalities, of the relationship between the joint distribution over potential outcomes and the observed data distribution under IV models when the instrument, treatment and outcome take finitely many values.  
In fact, we show that our characterization applies to five different IV models defined in terms of different versions of an independence condition and an exclusion restriction. The set of inequalities we obtain are necessary, sufficient, and non-redundant.
Our proof of sufficiency is based on Strassen's Theorem, thereby circumventing the whole machinery of random set theory and capacities that have been employed in other analyses (see, e.g., \citealp{beresteanu,Russell}). This also leads to a self-contained proof of non-redundancy.
 
The linear characterization enables us to compute bounds on average treatment effects via linear programming. Importantly, the size of the set of inequalities grows linearly with the size of the state-space of the instrument. This reflects a one-to-one correspondence between the inequalities arising from the observed distributions in any two different $Z$ arms. For corresponding inequalities, the associated hyper-planes are all parallel. In practice this means that determining bounds on a pairwise ATE when the instrument takes $Q$ levels is computationally no harder than doing so given a single $Z$ arm (or, equivalently, given the joint distribution of treatment and outcome from an observational study).

The characterization also leads to a sharp falsification test: an observed distribution is compatible with any of the five instrumental variable models that we consider if and only if there is a solution to the set of linear inequalities given by our characterization.

Further, we show that by solving a convex program, one can construct an interval that contains the upper and lower bound on any linear functional of the joint counterfactual distribution with a (conservative) finite-sample coverage guarantee. {Falsification is incorporated directly into this confidence-region construction, thereby unifying model assessment and statistical inference within a single procedure.} The convex program is formed by supplementing the linear program arising from our characterization with additional constraints relating the observed population distribution to the empirical distribution given by a finite-sample tail bound on the Kullback--Leibler divergence under multinomial sampling \citep{GuoChernoff}.

\vspace{-0.5cm}\subsection{Related prior work}\vspace{-0.3cm}

IV models with binary instrument, treatment, and outcome have been well-studied. \cite{robins:1989:ivbounds}, \cite{manski1990}, \cite{BPbounds}, and \cite{RichardsonRobins} derived sharp lower and upper bounds on the ATE under different versions of the independence and exclusion restriction conditions; see \cite{swansonreview} for a comprehensive discussion.   \cite{RichardsonRobins} 
extended these results by
showing that when the instrument takes $Q$ levels, but treatment and outcome are binary, the joint distribution over the potential outcomes $P(Y(x_1),Y(x_2))$
is characterized by a set of $8Q$ inequalities.
This characterization leads to simple closed form expressions for bounds on the ATE.

\cite{beresteanu} use random set theory, and in particular, Artstein's Theorem, to provide a characterization of the joint distribution of the potential outcomes in an instrumental variable model, where the treatment takes finitely many values, while the instrument and outcome take values in a compact subset of $\mathbb{R}$. Though the characterization is elegant, the resulting set of inequalities can be computationally prohibitive, with its size growing super-exponentially in the number of levels of treatment since there is one inequality for every (non-trivial) subset of joint values taken by the vector of potential outcomes. More precisely, if the treatment $X$ and outcome $Y$ take $K$ and $M$ levels respectively, then the vector of potential outcomes $(Y(x_1),\ldots, Y(x_K))$ takes $M^K$ different values; thus the result requires $Q(2^{\left(M^K\right)}-2)$ inequalities. For example, when $X$ and $Y$ both take $3$ values, Artstein's Theorem yields $ (2^{27}-2)\cdot Q > 10^8\cdot Q$ inequalities, whereas it follows from \cref{cor:numbounds} below that in fact only $333\cdot Q$ are required! {Thus, our set of inequalities are different from those of \cite{beresteanu}, unless $X$ and  $Y$ are binary; see \cref{tab:num-ineqs-km-sgcr-bmm}.}

{\cite{balke:1995} and \cite{gabriel:2025} give closed-form bounds on the ATE when $X$ is ternary, with $Z$ and $Y$ binary. They further showed that the classical Balke--Pearl bounds for a binary treatment remain valid when applied to a three-level treatment by plugging in the probabilities for the two treatment levels of interest as observed in the full three-level distribution. However, these bounds are generally not sharp, as the sharp three-level bounds incorporate additional constraints involving the third treatment level and can therefore produce strictly narrower bounds. }

\citet{chesher:rosen:2017} and \citet{Russell} noted previously that the set of inequalities resulting from a direct application of Artstein's Theorem was larger than required. 
\citet{LuoWang} give a general characterization of a subset of non-redundant inequalities implied by Artstein's Theorem, while \cite{ponomarev2026sharp} further extended the results.
\citet{Russell} gave a set of inequalities that he states result from applying the characterization of \citet{LuoWang} to the IV model.
In Supplement \ref{appendix:russell} we show that, in general, the set of inequalities described by Russell is too small.
Thus, the resulting inequalities may include joint distributions that are incompatible with the IV models and can fail to provide a sharp bound for functionals of the joint counterfactual distribution.

Other authors have addressed the question of whether a given observed distribution is compatible with particular sets of IV assumptions. 
\citet{pearl1994} introduces an ``instrumental inequality'' that provides a necessary condition, thus providing a falsification test. {For the binary instrumental variable model, \citet{wang2017falsification} further developed finite-sample inference procedures for these instrumental inequalities, enabling formal falsification tests with controlled Type I error.} Applying polyhedral geometry, 
\citet{joris:berkeley-admissions} show that when $K\geq 2$, while $Q=M=2$, the IV inequalities define the observed model. In contrast, \citet{bonet} showed that when $Q=3$ and $K=M=2$ there are additional inequalities. 
\citet{KedagniMourifie} 
proposed a generalized set of inequalities that are necessary for an observed distribution to be compatible with the IV model defined by
Individual-level exclusion and Randomization; see A\ref{assumption:exclusion}-1, 
A\ref{assumption:indep}-1 below. They further showed that when $K=M=2$ these inequalities are also sufficient and thus define the model for the observed distribution. 
Our results generalize this result by showing that five different formulations of the IV model all lead to the same set of observed distributions. {Additionally, no sharpness results beyond $K=M=2$ were established in \cite{KedagniMourifie}. 
Although the authors considered settings with non-binary instrument, treatments and outcomes, we provide a direct comparison demonstrating that their bounds are generally not sharp (Supplementary \cref{supp:KM}).}

\citet[\S8.4]{pearl:2000} notes that in the case of a binary IV model the instrumental 
inequality arises from requiring that
the bounds on the ATE are non-empty. Similarly, our characterization in \cref{theo:redundancy} provides a sharp test in that an observed distribution is compatible with the model if and only if the set of inequalities implies a non-empty set of distributions for the potential outcomes.

Most of the literature on delivering inference for a partially identified treatment effect $\tau$ in the IV setting employs methods for conditional moment inequalities {\protect{\citep{andrews2013inference}}} or intersection bounds \citep{chernozhukov2013intersection}. Both methods require obtaining relatively explicit bounds for $\tau$, in the form of $\{\tau: g_P(\tau, v) \leq 0 \text{ for all } v \in V\}$ for the former ($g_P(\tau,v)$ is a conditional moment indexed by $v$) and the form of $\tau \in (\sup_{v \in V} l_P(v), \inf_{v \in V} u_P(v))$ for the latter; the reader is referred to \citet{canay2017practical,shi2025inference} for surveys on these methods. Sophisticated bootstrap \citep{ramsahai2011likelihood} and Monte-Carlo \citep{sachs2025improved} methods have also been considered for binary IV. Similar to the inference approach in this paper, \citet{duarte2024automated} constructs a confidence region for the observed distribution in the discrete setting, which is then incorporated into a polynomial program for computing confidence intervals. For Bayesian methods, see also \citet{richardson2011transparent,silva2016causal}.

 \begin{figure}
 		\centering
 		\begin{tikzpicture}
 			\node[state] (x) at (0,0) {$X$};
 			\node[state] (z) [left =of x, xshift=-5cm] {$Z$};
 			\node[state] (y) [right =of x, xshift=2cm] {$Y$};
 			\node[state] (u) [above right =of x,xshift=0.3cm,yshift=0.5cm] {$U$};
 			
			\path (z) -- (y) node[midway, above, yshift=-1.85cm]
      {\textcolor{red}{\XSolidBrush}};
             \path (z) -- (u) node[midway, above, yshift=-0.35cm]
      {\textcolor{red}{\XSolidBrush}};

                \draw[<->, dashed, thick] (z) -- (u) node [midway, above, sloped, font=\footnotesize, yshift=0.1cm] {(1) Independence condition};
                
                \draw [->, dashed, looseness=0.8, thick] (z) to [out=-30,in=-150] node [midway, below, sloped, font=\footnotesize, yshift=-0.15cm]{(2) Exclusion Restriction}(y);
                
                \draw [->, thick] (z) to  (x) node [below, xshift = -3cm, font=\footnotesize] { Relevance condition};
                
                \draw [->, thick] (x) to  (y);
                
                \draw [->, thick] (u) to  (x);
                
                \draw [->, thick] (u) to  (y);
 		\end{tikzpicture} 
   \centering
    \caption{Directed acyclic graph (DAG) representing the assumptions of a valid instrumental variable, where the dashed edges are assumed to be absent. }
    \label{fig:dag_assump}
 \end{figure}

\vspace{-0.6cm}\subsection{Outline} \vspace{-0.3cm}

In \cref{sec:notation}, we introduce our notation and assumptions and present five instrumental variable models where our results apply. In \cref{sec:main-res}, we present our main theorems on the characterization of the joint probability distribution of the potential outcomes. We lay out the general setup and the ingredients essential to the proof of our main results in \cref{sec:sufficiency,sec:redundancy}; additional details are presented in the Supplementary Materials. In \cref{sec:inference}, we discuss statistical inference based on a finite-sample tail bound of Kullback--Leibler divergence.
In \cref{sec:real-data}, we illustrate our method with real data from the Minneapolis Domestic Violence Experiment, where the instrument and treatment both take three levels. 
Finally, we conclude our paper with discussion of future work in \cref{sec:discussion}.

\vspace{-0.6cm} \section{Notation, assumptions, and models}\label{sec:notation} \vspace{-0.3cm}
Consider a categorical outcome $Y$ with $M \geq 2$ levels, a treatment variable $X$ with $K \geq 2$ levels, and an instrumental variable $Z$ with $Q \geq 1$ levels. The setting of a single $Z$-arm, where $Q=1$, corresponds to an observational study. 

We assume $Y,X,Z$ take values in $[M], [K], [Q]$ respectively. Here we use the shorthand $[M] := \{1,\dots,M\}$ and similarly for other integers.
When it is clear from context that $Z$ is fixed to $z \in [Q]$, we will often omit the conditioning in $P(\cdot \mid Z=z)$.
For $k \geq 1$, we use $\Delta^{k-1} \subset \mathbb{R}^k$ to denote the $(k-1)$-dimensional probability simplex. For a set $A$, we use $\overline{A}$ to denote its complement. We use $A \subseteq B$ (or $B \supseteq A$) to denote that $A$ is a subset of $B$; we use $A \subset B$ (or $B \supset A$) to denote that $A$ is a proper subset of $B$. 
We use `$=_{d}$' to denote equality in distribution or conditional distribution. We use $\delta_x$ to denote a point mass at $x$. 

\vspace{-0.5cm}\subsection{Assumptions}\label{sec:assumptions}\vspace{-0.3cm}

For all the models we consider, we will assume the existence of potential outcomes $Y(x=i,z=q)$ for $i\in [K]$, $q\in [Q]$, corresponding to the value of $Y$ for a randomly selected subject if the subject were to receive $Z=q$ and $X=i$ (possibly counter-to-fact). In addition, for certain models we also assume the existence of potential outcomes $X(z=q)$ for $q \in [Q]$, denoting the value of the treatment $X$ that a subject would receive had the subject been assigned to $Z=q$. We will often use the shorthand notation $Y(x_i,z_q) := Y(x=i,z=q)$ and 
$X(z_q) := X(z=q)$. The observed data and potential outcomes are related via the usual consistency relation: $Y=Y(X,Z)$; 
for models with $X(z)$ potential outcomes, we also have $X=X(Z)$. We also define $Y(x) := Y(x,Z)$ to be the potential outcome for $Y$ arising from an intervention on $X$ alone. We will also consider a latent variable formulation, which posits the existence of an unmeasured variable $U$ (with unknown state-space) that represents all variables giving rise to the confounding between $X$ and $Y$.

The instrumental variable model is based on an Exclusion assumption and an Independence assumption. Different forms of these have been considered in the literature (see, e.g., \citealp[\S5.1.2]{guo2021likelihood}):

\begin{assumption}[Versions of the Exclusion assumption]\label{assumption:exclusion} \hfill
\begin{enumerate}[label=(A\ref{assumption:indep}-*)]
    \item[(A\ref{assumption:exclusion}-1)] 
    Individual-level exclusion: \vspace{-0.3cm}
    \begin{equation} \label{eq:individual-exclusion}
    \hspace{-2em} Y(x_i,z)=Y(x_i, \tilde{z}) \, \text{ almost surely for all 
    $z,\tilde{z} \in [Q]$ and every $i \in [K]$}.
\end{equation}
    \item[(A\ref{assumption:exclusion}-2)\label{assump:joint} ] \vspace{-0.5cm} Joint stochastic exclusion: \vspace{-0.3cm}
    \begin{equation}\label{eq:joint-stoc-exclusion}
    \hspace{-2em} \left(Y(x_1,z), \dots ,Y(x_K,z)\right) =_{d} (Y(x_1,\tilde{z}), \dots, Y(x_K,\tilde{z})) \\
\;\text{ for all $z,\tilde{z} \in [Q]$}.
    \end{equation}
    
    \item[(A\ref{assumption:exclusion}-3)] \vspace{-0.5cm} Latent stochastic exclusion:  \vspace{-0.3cm}
\begin{equation}\label{eq:latent-exclusion}
    Y(x,z) \mid U \;\;\, =_{d} \;\;\, Y(x,\tilde{z}) \mid U \;\;\;\;\;\text{ for all $z,\tilde{z} \in [Q]$ and every $x \in [K]$}.
\end{equation}
\end{enumerate}
\end{assumption}

\vspace{-0.5cm}
The strongest version (A\ref{assumption:exclusion}-1) requires that there is no direct effect of $Z$ on $Y$ relative to $X$ at the individual level.
The weaker versions (A\ref{assumption:exclusion}-2) and (A\ref{assumption:exclusion}-3) restrict the effect of $Z$ on $Y$ relative to $X$ at the population level. Specifically, version (A\ref{assumption:exclusion}-3) means that the direct effect of $Z$ on $Y$ holding $X$ and a latent variable $U$ fixed is zero at the population level.
The joint stochastic exclusion assumption (A\ref{assumption:exclusion}-2) generalizes a condition given in \citet{swansonreview}; 
\cite{hirano2000} also consider a related stochastic exclusion assumption.

Different independence assumptions have also been considered; see \citet{swansonreview} for a review focused on the binary IV model. We consider the following versions:
\begin{assumption}[Versions of the Independence assumption]\label{assumption:indep} 
\hfill
\begin{enumerate}[label=(A\ref{assumption:indep}-*)]
    \item[(A\ref{assumption:indep}-1)] \vspace{-0.5cm} Random assignment: \vspace{-0.3cm}
    \begin{equation}
    Z\ind \left\{Y(x,z),X(z):\, x\in [K], z\in [Q] \right\}.
    \label{eq:random-assignment}
    \end{equation}
    \item[(A\ref{assumption:indep}-2)] \vspace{-0.5cm} Joint independence: \vspace{-0.3cm}
    \begin{equation}
    Z\ind \left\{Y(x,z): \, x\in [K], z\in [Q] \right\}.
    \label{eq:joint-indep}
    \end{equation}
    \item[(A\ref{assumption:indep}-3)] \vspace{-0.5cm} {Single-world independence}: \vspace{-0.3cm}
    \begin{equation}
\text{ for all $z\in [Q]$, $x\in [K]$,}\quad    Z\ind X(z),\, Y(x,z).
    \label{eq:single-world}
    \end{equation}
    
    \item[(A\ref{assumption:indep}-4)] \vspace{-0.5cm} Latent-variable independence: there exists $U$ such that   \vspace{-0.3cm}
    \begin{equation}    \label{eq:latent-exogeneity}
U \ind Z \;\text{ and for all $z\in [Q]$, $x\in [K]$,}
\;\; Y(x,z)\ind X,Z \mid U.
    \end{equation}
\end{enumerate}
\end{assumption}

\vspace{-0.5cm} In the binary setting where $Q=K=M=2$ the Balke--Pearl bounds were derived under (A\ref{assumption:exclusion}-1) and (A\ref{assumption:indep}-1) but are shown to also hold under the weaker independence assumptions (A\ref{assumption:indep}-2), (A\ref{assumption:indep}-3), and (A\ref{assumption:indep}-4); see \cite{RichardsonRobins}. \cite{kitagawa2021} analyzed the IV model under (A\ref{assumption:indep}-2). Other works including \cite{dawid2003} formulated the IV model with the presence of an unmeasured confounder $U$ between $X$ and $Y$ as defined in (A\ref{assumption:indep}-4). \cite{RichardsonRobins} developed a sharp characterization of the joint counterfactual probability distribution $P(Y(x_1), Y(x_2))$ given an observed conditional probability $P(X,Y\mid Z)$, which holds under any of the independence conditions (A\ref{assumption:indep}-1)--(A\ref{assumption:indep}-4). 

 \begin{figure}[!htb]
 		\centering
 		\begin{tikzpicture}
 			\node (M1) at (0,0) {${\calm}_1$};
 			\node (M2) at (-1,-1) {${\calm}_2$};
 			\node (M3) at (-2, -2) {${\calm}_3$};
                \node (M4) at (1, -1) {${\calm}_4$};
 			\node (M5) at (0, -2) {${\calm}_5$};      
                \node at ($(M1)!0.5!(M2)$) [rotate=45] {$\supset$};
                \node at ($(M1)!0.5!(M4)$) [rotate=-45] {$\subset$};
                \node at ($(M2)!0.5!(M3)$) [rotate=45] {$\supset$};
                \node at ($(M2)!0.5!(M5)$) [rotate=-45] {$\subset$};
 		\end{tikzpicture} 
   \centering
    \caption{Nested structure between models ${\calm}_1$--${\calm}_5$.}
    \label{fig:model}
 \end{figure}
 
\begin{table}[tb]
\caption{Instrumental variable models considered in this paper}\label{tab:models}
\begin{tabular}{llcc}
& {\bf Model Name} & {\bf Exclusion} & {\bf Independence}\\
\hline
${\calm}_1$ & {\it Randomization}   & Individual-level & Random assignment\\
${\calm}_2$ & {\it Joint Ind. \& Indiv.~Excl.} & Individual-level & Joint independence\\
${\calm}_3$ & {\it Joint Ind. \& Stoch. Excl.} & Joint stochastic exclusion & Joint independence\\
${\calm}_4$ & {\it SWIG}  & Individual-level & Single-world independence\\
${\calm}_5$ & {\it Latent Model} & Latent stochastic exclusion & Latent-variable independence\\
\end{tabular}
\end{table}

\vspace{-0.5cm} In this paper, we consider five models $\calm_1,\dots, \calm_5$ corresponding to five different combinations of the exclusion and independence assumptions as shown in \cref{tab:models}.\footnote{In the setting where
$M=K=2$, \cite{RichardsonRobins} consider the models 
$\calm_1$, $\calm_2$, $\calm_4$ and another model $\calm_5^*$
given by (A\ref{assumption:exclusion}-1) and
(A\ref{assumption:indep}-4). Since
$\calm_2\subset \calm_5^*\subset \calm_5$, our results also apply to this model.} \cref{fig:graphs} displays the graphical models corresponding to $\calm_1,\dots,\calm_5$. We describe the relationship among the models $\calm_1,\dots,\calm_5$ in \cref{fig:model} and \cref{lemma:models} below.

\begin{lemma}\label{lemma:models}
    We have $\calm_1 \subset \calm_2 \subset \calm_3$, $\calm_1 \subset \calm_4$
    and $\calm_2 \subset \calm_5$.
\end{lemma}

\vspace{-0.5cm}
\begin{proof}
First, observe that individual-level exclusion implies both joint stochastic exclusion and latent stochastic exclusion. Second, observe that random assignment implies both joint independence and single-world independence. Finally, observe that joint independence implies latent-variable independence as we can set the latent variable to $U:=(Y(x,z): x \in [K], z \in [Q])$. The result then follows from the definition of the models in \cref{tab:models}.
\end{proof}

{Our analysis applies to models $\calm_2$ and  $\calm_3$ that posit neither potential outcomes for $X$ nor latent variables. Yet, these models require stronger exclusion assumptions; see also \citet[\S5.3]{guo2021likelihood}.}
Weaker versions of the instrumental variable model based on marginal independence assumptions have been considered by many authors \citep{beresteanu,kitagawa2021,manski1990,robins:1989:ivbounds}. In general, these are strict supermodels and imply wider bounds on the ATE.


\begin{figure}[htbp]
\begin{center}
\begin{tikzpicture}[>=stealth, ultra thick, node distance=2cm, inner sep=1.5,
    pre/.style={->,>=stealth,ultra thick,blue,line width = 1.2pt}]
\begin{scope}
\node[thick, name=z,shape=circle,style={draw}]
{$Z$
};
\node[above =1.6cm of z.center]{(a) ${\cal M}_1$};
\node[thick,name=x,shape=circle,style={draw}, right=0.8cm of z] {
$X$
};
\node[thick, name=y,shape=circle,style={draw},right= 2.2cm of x]
{$Y$
};
\node[thick, name=xcounter,shape=ellipse,style={draw},above=0.4cm of z, inner sep=1pt,xshift=2mm]
{\small $\{X(z)\}_{{z\in [Q]}}$
};
\node[thick, name=ycounter,shape=ellipse,style={draw},above=0.4cm of y, inner sep=1pt]
{\small $\{Y(x,\!z)\}_{\tiny x \in [K],z\in [Q]}$
};
\draw[pre,->,double,thick] (x) to (y);
\draw[pre,<->] (xcounter) to[out=30,in=150] (ycounter);
\draw[pre,->,double, thick] (xcounter) to (x);
\draw[pre,->,double, thick] (ycounter) to (y);
\draw[pre,->,double, thick] (z) to (x);
\end{scope}
\begin{scope}[yshift=-2.7cm]
\node[thick, name=z,shape=circle,style={draw}]
{$Z$
};
\node[above =1.4cm of z.center]{(b) ${\cal M}_2$};
\node[thick,name=x,shape=circle,style={draw}, right=0.8cm of z] {
$X$
};
\node[thick, name=y,shape=circle,style={draw},right= 2.2cm of x]
{$Y$
};
\node[thick, name=u,shape=ellipse,style={draw},above=0.4cm of z, xshift=14pt, inner sep=1.5pt]
{$U^*$
};
\node[thick, name=ycounter,shape=ellipse,style={draw},above=0.4cm of y, inner sep=1pt]
{\small $\{Y(x,\!z)\}_{\tiny x \in [K],z\in [Q]}$
};
\draw[pre,->,double,thick] (x) to (y);
\draw[pre,->] (u) to (x);
\draw[pre,->] (u) to (z);
\draw[pre,->] (ycounter.200) to (x.45);
\draw[pre,->,double, thick] (ycounter) to (y);
\draw[pre,->] (z) to (x);
\end{scope}
\begin{scope}[yshift=-5.5cm]
\node[thick, name=z,shape=circle,style={draw}]
{$Z$
};
\node[above =1.5cm of z.center]{(c) ${\cal M}_3$};
\node[thick,name=x,shape=circle,style={draw}, right=0.8cm of z] {
$X$
};
\node[thick, name=y,shape=circle,style={draw},right= 2.2cm of x]
{$Y$
};
\node[thick, name=ustar,shape=ellipse,style={draw},above=0.4cm of z, xshift=14pt, inner sep=1.5pt]
{$U^*$
};
\node[thick, name=ycounter,shape=ellipse,style={draw},above=0.4cm of y, inner sep=1pt]
{\small $\{Y(x,z)\}_{\tiny x \in [K], z\in[Q]}$
};
\node[below=0.1cm of x,xshift=1.5cm]{ \small \& $P\left( \{Y(x,z)\}_{\tiny x \in [K] }\right) =  P\left( \{Y(x,z^*)\}_{\tiny x \in [K]} \right)$};
\draw[pre,->] (x) to (y);
\draw[pre,->] (ustar) to (x);
\draw[pre,->] (ustar) to (z);
\draw[pre,->] (ycounter.200) to (x.45);
\draw[pre,->] (ycounter) to (y);
\draw[pre,->] (z) to (x);
\end{scope}
\begin{scope}[xshift=8cm,yshift=0.5cm]
\begin{scope}
			\tikzset{line width=0.9pt, inner sep=1.8pt, swig vsplit={gap=4pt, inner line width right=0.3pt}}	
				\node[xshift=0.0cm, yshift=0.0cm, name=z, shape=swig vsplit]{
        					\nodepart{left}{$Z$}
        					\nodepart{right}{$\color{red}{z}$} };
		\end{scope}
\node[above =1.0cm of z.center]{(d) ${\cal M}_4$};
\begin{scope}
			\tikzset{line width=0.9pt, inner sep=1.8pt, swig vsplit={gap=4pt, inner line width right=0.3pt}}	
				\node[xshift=0.0cm, yshift=0.0cm, name=x, shape=swig vsplit, right=0.8cm of z]{
        					\nodepart{left}{$X(z)$}
        					\nodepart{right}{$\color{red}{x}$} };
		\end{scope}
\node[thick, name=y,shape=ellipse,style={draw},right= 1.2cm of x]
{$Y(x)$
};
\draw[pre,->] (x) to (y);
\draw[pre,<->] (x) to[out=100,in=100, looseness=0.9] (y);
\draw[pre,->] (z) to (x);
\end{scope}
%
%
\begin{scope}[xshift=8cm,yshift=-3.3cm]
\begin{scope}
			\tikzset{line width=0.9pt, inner sep=1.8pt, swig vsplit={gap=4pt, inner line width right=0.3pt}}	
				\node[xshift=0.0cm, yshift=0.0cm, name=z, shape=swig vsplit]{
        					\nodepart{left}{$Z$}
        					\nodepart{right}{$\color{red}{z}$} };
		\end{scope}
\node[above =2cm of z.center]{(e) ${\cal M}_5$};
\begin{scope}
			\tikzset{line width=0.9pt, inner sep=1.8pt, swig vsplit={gap=4pt, inner line width right=0.3pt}}	
				\node[xshift=0.0cm, yshift=0.0cm, name=x, shape=swig vsplit, right=0.8cm of z]{
        					\nodepart{left}{$X(z)$}
        					\nodepart{right}{$\color{red}{x}$} };
		\end{scope}
\node[thick, name=y,shape=ellipse,style={draw},right= 1cm of x]
{$Y(x,z)$
};
\node[name=ca,right =0.4cm of x]{};
\node[thick,name=u,above=1.3cm of ca,shape=circle,style={draw}] {
$U$
};
\node[thick,name=ustar,above right =1.2cm and 0.3cm of z,shape=ellipse,style={draw}] {
$U^*$
};
\draw[pre,->,thick] (x) to (y);
\draw[pre,->] (u) to[out=260, in=85, looseness=0.9] (x);
\draw[pre,->] (ustar) to (x);
\draw[pre,->] (ustar) to[out=260,in=100, looseness=0.9] (z);
\draw[pre,->] (u) to (y);
\draw[pre,->, thick] (z) to (x);
\end{scope}
\end{tikzpicture}
\end{center}
\caption{
Graphical representations of independence and exclusion assumptions discussed in Section~\ref{sec:assumptions}.  $\calm_1$ and $\calm_4$ do not have confounding between $Z$ and $X$ and independence is encoded using the extension of d-separation  to acyclic graphs with bi-directed ($\leftrightarrow$) edges \citep{richardson-admg:2003}; $\calm_2$,
$\calm_3$ and $\calm_5$ allow confounding between $Z$ and $X$ and their independence assumptions follow
from Pearl's d-separation for directed acyclic graphs.  (Note that {\rm (e)} encodes a
slightly stronger version of (A\ref{assumption:indep}-4) with $X(z)$ replacing $X$.)
In {\rm (a)} and {\rm (b)}, when a variable is connected to its parents with double edges ($\Rightarrow$), the variable is a deterministic function of its parents.
The individual exclusion assumption
(A\ref{assumption:exclusion}-1) in  $\calm_1$ and $\calm_2$
follows because $Y$ is determined by $\{Y(x,z)\}$ and $X$ (and not $Z$);  The joint stochastic exclusion assumption  
(A\ref{assumption:exclusion}-2) in $\calm_3$ cannot be (easily) represented graphically and is stated explicitly; individual exclusion in $M_4$ is implied because the SWIG contains $Y(x)$ rather than $Y(x,z)$; the  latent stochastic exclusion assumption in $\calm_5$  is signified by the absence of an edge from $z$ to $Y(x,z)$; indeed, it holds that $z$ is d-separated from $Y(x,z)$ given $U$ \citep{malinsky19b,RichardsonRobins-2023}.}
\label{fig:graphs}
\end{figure}


\vspace{-0.6cm} \section{Main results}\label{sec:main-res}\vspace{-0.3cm}

We characterize the joint counterfactual distribution for the potential outcomes of $Y$. 

\begin{theorem}\label{theorem:main-result}
Under each of the models ${\calm}_1, \dots, {\calm}_5$, the relationship between the observed distribution 
$P(X,Y\,|\, Z)$ and the 
joint counterfactual probability distribution $P'(Y(x_1),\dots, Y(x_K))$ is characterized by the same set of inequalities: 
\begin{equation} \label{theo:1cont}
          P'\left(Y(x_1) \in {\calv}^{(1)},\ldots,Y(x_K)\in {\calv}^{(K)}\right)\;\leq\; \sum_{i=1}^K P\left(\left.X\!=\!i, Y\in {\calv}^{(i)}\right| Z\!=\!z\right), \; z \in [Q],
\end{equation}
where ${\calv}^{(k)}$ is a non-empty subset of $[M]$ for every $k\in [K]$ and a strict subset of $[M]$ for at least one ${k}$. There are $Q((2^M-1)^K-1)$ such inequalities. 
\end{theorem}

The inequalities (\ref{theo:1cont}) are {\it necessary} in that they are implied by each of the models ${\calm}_1$, \ldots, ${\calm}_5$. The set of inequalities are also {\it sufficient}: given any counterfactual distribution
$P^\prime(Y(x_1),\dots, Y(x_K))$
and any observed distribution $P(X,Y\,|\, Z)$ obeying (\ref{theo:1cont}), there exists a joint distribution $\check{P}\big(Z,X, Y(x_1), \ldots, Y(x_K)\big)$
that has margins $P'$ and $P$ 
and is compatible with each of the models ${\calm}_1$, \ldots, ${\calm}_5$.

\Cref{theo:1cont} consists of $Q((2^M-1)^K-1)$ inequalities: here $2^M-1$ counts the non-empty subsets of $[M]$; the second `$-1$' arises from the requirement that at least one ${\calv}^{(k)}$
be a strict subset (otherwise the inequality becomes trivial, since both sides are $1$). We further note that both the left- and right-hand sides of all bounds in the form of (\ref{theo:1cont}) are linear summations of $P'(Y(x_1)=y^{1},\dots, Y(x_K)=y^{K})$ and $P(Y=y, X=x\mid Z=z)$. This makes the practical implementation of our bounds efficient. 

\begin{remark}
Even though \cref{theorem:main-result} is formulated as upper-bounding the joint counterfactual probabilities with observed probabilities, the set of inequalities implies both upper and lower bounds for any joint counterfactual probability by normalization of the probability measure.
\end{remark}

\begin{remark} \label{rem:falsify}
For a given observed distribution $P(X,Y\,|\,Z)$, the inequalities \eqref{theo:1cont} describe a polytope for $P'$, which is the set of counterfactual distributions compatible with $P$ under any of the IV models we consider. If this set of $P'$ is empty, then $P$ must lie outside the models $\calm_1, \dots, \calm_5$ and hence the IV is falsified. If desired, using Fourier--Motzkin to eliminate $P'$ from \eqref{theo:1cont}, one can obtain a set of inequalities on $P$ that directly describe the set of observed distributions compatible with the IV models, which generalize the instrumental inequalities in \citet{BPbounds,bonet,KedagniMourifie}; alternatively, one can check feasibility computationally, which we discuss in Supplemental Section~\ref{sec:feasibility}. In \cref{sec:inference}, we will describe an inference algorithm that incorporates model falsification test without explicitly requiring these instrumental inequalities. 
\end{remark}

When we specialize \eqref{theo:1cont} by setting $\calv^{(i)}$ to be either $[M]$ or a singleton $\{j_{(i)}\}$ for every $i \in [K]$, we obtain the following upper bounds on marginal counterfactual probabilities. 

\begin{corollary}\label{coro:marginal}
  The following inequalities follow from \eqref{theo:1cont}:
{\small \begin{align*} 
P'(Y(x_i)=j) & \leq 1-P(X=i, Y\neq j\mid Z=z),\\
P'\left(Y(x_i)=j, Y(x_{i'})=j'\right) &\leq 1-P(X=i, Y\neq j\mid Z=z)-P\left(\left. X=i', Y\neq j'\right| Z=z\right),\\[-4pt]
& \;\;\vdots \\[-4pt]
P'\left( Y(x_{i_{(1)}})=j_{(1)}, \ldots , Y(x_{i_{(k)}})=j_{(k)}\right) & \leq 1-P\left(\left.X=i_{(1)}, Y\neq j_{(1)}\right| Z=z\right)-\cdots \nonumber \\[-6pt]
        & \kern50pt \qquad \quad - P\left(\left.X=i_{(k)}, Y\neq j_{(k)}\right| Z=z\right),\\[-16pt]
& \;\;\vdots \\[-4pt]
P'\left( Y(x_{i_{(1)}})=j_{(1)}, \ldots , Y(x_{i_{(K)}})=j_{(K)}\right) & \leq 1-P\left(\left.X=i_{(1)}, Y\neq j_{(1)}\right| Z=z\right)-\cdots \nonumber \\[-6pt]
        & \kern50pt \qquad \quad - P\left(\left.X=i_{(K)}, Y\neq j_{(K)}\right| Z=z\right),
\end{align*}}
where $z \in [Q]$, $1 \leq i_{(1)} < \dots < i_{(k)} \leq K$, and 
         $(j_{(1)},\dots, j_{(k)}) \in [M]^k$.
\end{corollary}
         
In the special case of $M=2$ (binary $Y$), these are exactly the same inequalities as in \cref{theo:1cont}. When $M>2$, there are additional inequalities in \cref{theo:1cont} which are not bounds on the marginalized counterfactual probabilities.

For any given instrument arm $z$, the set\footnote{Here, `set' implies no two inequalities are identical.} of inequalities (\ref{theo:1cont}) specified in \cref{theorem:main-result} define a finite polytope over the pairs $(P'(Y(x_1), \dots, Y(x_k)), P(X,Y \mid Z=z))$ of counterfactual and observed distributions. Relative to a given set of inequalities, we call an individual inequality \emph{redundant} if it is implied by the remaining inequalities in the set. A set of inequalities is called \emph{non-redundant} if no individual inequality is redundant relative to the set. By a basic result on finite polytopes \citep[Theorem 2.15]{polytope-book}, an inequality is not redundant if and only if the half-space defined by the inequality is facet-defining, i.e., the inequality can hold with equality for some point in the polytope and when the inequality holds with equality, the resulting hyperplane corresponds to a facet (a face with maximum dimension) of the polytope. Obtaining a non-redundant set of inequalities to characterize the polytope is essential for reducing the complexity in describing the model and computing partial identification bounds. For example, under regularity conditions, interior-point methods for convex optimization achieve time complexity that is polynomial in the number of inequalities \citep{nesterov1994interior}.

\begin{theorem}\label{theo:redundancy}
The set of inequalities \eqref{theo:1cont} can be reduced to a subset that only consists of inequalities that satisfy either \vspace{-0.3cm}
    \begin{enumerate}
        \item \label{cond:redundancy-1} for at least two values $k \neq k^*$, we have ${\calv}^{(k)}\neq [M]$ and ${\calv}^{(k^*)}\neq [M]$,
        \vspace{-0.3cm}
        
or \vspace{-0.3cm}

        \item  \label{cond:redundancy-2} there exist $k^{\ast}$ and $m \in [M]$ such that ${\calv}^{(k^{\ast})}=[M]\setminus\{m\}$ and ${\calv}^{(k)}=[M]$ for every $k \neq k^{\ast}$.
    \end{enumerate}
This subset of inequalities are equivalent to \eqref{theo:1cont} and non-redundant. Compared to \eqref{theo:1cont}, this subset has $Q(K(2^M-M-2))$ fewer inequalities. 
\end{theorem}

Under Condition 2 above, \cref{theo:1cont} becomes $P'(Y(x_k)\neq m)\leq 1-P(X=k, Y=m\mid Z=z)$. The inequalities that are redundant, i.e., those that satisfy neither Condition 1 nor 2, thus take the form \vspace{-0.5cm}
\[ P'(Y(x_k)\not\in\{m_1,\dots, m_J\})\leq 1-\sum_{j=1}^J P(X=k, Y=m_j\mid Z=z), \quad J\geq 2. \]

\vspace{-0.5cm}\begin{corollary}\label{cor:numbounds} 
The subset of inequalities specified in \cref{theo:redundancy} consists of 
\begin{equation} \label{eqs:numbounds}
r=Q \, ((2^M-1)^K-K(2^M-M-2)-1)
\end{equation}
inequalities. This subset of inequalities is necessary, sufficient, non-trivial, and non-redundant for characterizing the pairs of compatible observed and joint counterfactual distributions under each of the models ${\calm}_1, \dots, {\calm}_5$.
\end{corollary}

\begin{remark}
In the case of $M=2$, the set of inequalities \eqref{theo:1cont} are non-redundant. This is because Condition 2 in \cref{theo:redundancy} is always satisfied since we know there is at least one $k^{\ast}$ such that $\calv^{(k^{\ast})}\neq\{1,2\}$ and $\calv^{(k^{\ast})}\neq \emptyset$. 
\end{remark}


\begin{example} Consider an IV model with
a binary treatment and ternary outcome, so 
$K=2$ and $M=3$. We first consider some non-redundant inequalities. Take ${\calv}^{(1)}=\{1,2,3\}$ and ${\calv}^{(2)}=\{1,2\}$. In this case \cref{theo:1cont}, namely 
\[ P'(Y(x_1)\in \calv^{(1)}, Y(x_2)\in \calv^{(2)}) \leq \sum_{i=1}^2 P\left(\left.X=i, Y\in \calv^{(i)}\right| Z=z\right), \]
becomes 
\begin{equation} \label{eq:nonredundantexp1}
P'(Y(x_2)\neq 3)\leq 1-P(X=2, Y=3\mid Z=z).
\end{equation}
This inequality is non-redundant because Condition 2 in \cref{theo:redundancy} is satisfied. 
Similarly, taking ${\calv}^{(1)}=\{1,2,3\}$ and ${\calv}^{(2)}=\{2,3\}$, gives 
\begin{equation} \label{eq:nonredundantexp2}
P'(Y(x_2)\neq 1)\leq 1-P(X=2, Y=1\mid Z=z),
\end{equation}
which is also non-redundant by \cref{theo:redundancy}.

In contrast, taking ${\calv}^{(1)}=\{1,2,3\}$ and ${\calv}^{(2)}=\{2\}$ gives the inequality 
\begin{equation} \label{eq:redundantexp}
P'(Y(x_2)=2)\leq 1-P(X=2, Y=1\mid Z=z)-P(X=2, Y=3\mid Z=z),
\end{equation}
which satisfies neither condition in \cref{theo:redundancy}. To see it is indeed redundant, note that \cref{eq:nonredundantexp1,eq:nonredundantexp2} can be rewritten as 
\begin{align*}
P'(Y(x_2)=3) &\geq P(X=2, Y=3\mid Z=z), \\
P'(Y(x_2)=1) &\geq P(X=2, Y=1\mid Z=z),
\end{align*}
which implies \cref{eq:redundantexp} by summing both sides. 

By enumerating all $({\calv}^{(1)}, {\calv}^{(2)})$ such that at least one of them is a strict subset of $\{1,2,3\}$, we can obtain the set of necessary, sufficient, and non-redundant inequalities that characterize $P'(Y(x_1), Y(x_2))$. By \cref{cor:numbounds}, the number of such inequalities is $42Q$.
\end{example}


Recall that each $\calm_i$ ($i=1,\dots,5$) is an IV model as specified in \cref{tab:models}. In what follows, we will overload the symbol $\calm_i$ to mean, specifically, the set of joint distributions $P(Z,X,Y(x_1),\dots, Y(x_K))$ under the model. 
We define 
\begin{equation} \label{eqs:phi}
\phi: P(Z,X,Y(x_1),\dots, Y(x_K))\mapsto \left(
\rule[-2pt]{0pt}{12pt}%
P(Y(x_1),\dots, Y(x_K)),\; P(X,Y\mid Z) \right),
\end{equation}
which maps the joint distribution of $Z$, $X$ and $Y$'s potential outcomes to the marginal distribution over the potential outcomes of $Y$ and the observed distribution of $X,Y$ given $Z$. The image of such a map is denoted by $\phi({\calm}_i)$. 
Let $\cal T$ denote the set of pairs of distributions $(P(Y(x_1),\dots, Y(x_K)), P(X,Y\mid Z))$ that obey the inequalities \eqref{theo:1cont}. Consequently, \cref{theorem:main-result} can be restated as $\phi({\calm}_i)={\cal T}$ for $i=1,\dots, 5$.


\vspace{-0.6cm}\section{Strassen's theorem and proof of sufficiency} \label{sec:sufficiency}\vspace{-0.3cm}
In this section, we prove the sufficiency of inequalities (\ref{theo:1cont}). The proof of necessity is in Supplemental~\cref{apx:proof-theo1}. 
Since $\calm_1$ is the smallest model (see \cref{lemma:models}), we only need to show ${\cal T}\subseteq \phi({\calm}_1)$. 
Our proof relies on Strassen's theorem \citep{Strassen1965}, which characterizes the condition for the existence of a probability measure with a given support and marginals. For our purposes, we use a finite-space version stated below due to \citet{koperbergthesis}. We will apply the theorem to each arm of $Z$, which characterizes the set of pairs $\left(P(Y(x_1),\dots, Y(x_K)), P(X,Y\mid Z=z) \right)$; then we will show that these characterizations for different $z$ can be combined to prove sufficiency.  

\begin{definition}[Neighbors] \label{def:neighbor}
    Let ${\cal A}$ and ${\cal B}$ be sets and ${\cal R}\subseteq {\cal A} \times {\cal B}$ a relation. Then for each $U\subseteq {\cal A}$, the set of neighbors of $U$ in ${\cal R}$ is \vspace{-0.3cm}
    $${\cal N}_{\cal R}(U):=\{\bmv\in {\cal B}: (U\times \{\bmv\})\cap {\cal R}\neq \emptyset\}.$$
\end{definition}
\vspace{-0.3cm}

\begin{definition}[Coupling]
    Let ${\cal A}$ and ${\cal B}$ be finite sets, $P_{\cal A}$ and $P_{\cal B}$ probability measures on ${\cal A}$ and ${\cal B}$ respectively. Then a coupling of $P_{\cal A}$ and $P_{\cal B}$ is a probability measure $\check{P}$ on ${\cal A} \times {\cal B}$, such that $\check{P}$ has $P_{\cal A}$ and $P_{\cal B}$ as marginals. 
\end{definition}

\begin{theorem}[Strassen's theorem for finite sets \citep{koperbergthesis}] \label{thm:strassen}
    Let ${\cal A}$ and ${\cal B}$ be finite sets, $P_{\cal A}$ and $P_{\cal B}$ probability measures on ${\cal A}$ and ${\cal B}$ and ${\cal R}\subseteq {\cal A}\times {\cal B}$ a relation. Then, there exists a coupling $\check{P}$ of $P_{\cal A}$ and $P_{\cal B}$ that satisfies $\check{P}({\cal R})=1$ if and only if 
\vspace{-0.3cm}\begin{equation} \label{eqs:strassen}
P_{\cal A}(U)\leq P_{\cal B}({\cal N}_{\cal R}(U)) \quad \text{for all } U\subseteq {\cal A}.
\end{equation}
\end{theorem}


\vspace{-0.3cm}To adapt the theorem to our case, we introduce some notation. 
We use $\mathcal{A}$ to denote the space of potential outcomes $(Y(x_1),\ldots,Y(x_K))$, given by $\mathcal{A} = [M]^K$.
Subsets of $\mathcal{A}$ describe events of potential outcomes. For example, when $K=3$, $\{(1,1,1)\} \subset \mathcal{A}$ denotes the event $\{Y(x_1)=1, Y(x_2)=1, Y(x_3)=1\}$. Let $\mathcal{B}$ denote the space of $(X,Y)$ so we have $\mathcal{B} = [K] \times [M]$.
Further, under the
individual-level exclusion assumption (assumed by $\calm_1$) and consistency, we have the following equivalence: \vspace{-0.3cm}
\[
(X=i,Y=y) \mid Z=z \quad\iff\quad (X(z)=i,Y(x_i)=y) \mid Z=z.
\]
Let us fix $z$. For any $\bm{a} \in \mathcal{A}$ and $\bm{b} \in \mathcal{B}$, we say $\bm{a}$ and $\bm{b}$ are \emph{coherent} if they assign the same value to any variable in common, or in other words, they obey consistency.
In light of the display above, we define the coherence relation ${\cal R}_{C} \subset \mathcal{A} \times \mathcal{B}$ as 
\begin{equation} \label{eqs:coherence}
\left(\bm{a}= (y^1, \dots, y^K),\bm{b}=(i,y)\right) \in {\cal R}_{C} \iff  y^i = y. 
\end{equation}
We can view ${\cal R}_{C}$ as specifying a set of edges in a bipartite graph; see \cref{fig:coherent} for the case of binary exposure and binary outcome.
For $(\bm{a},\bm{b}) \in {\cal R}_{C}$, the conjunction of $\bm{a}$ and $\bm{b}$ under $Z=z$ corresponds to an assignment to the whole vector $(X(z), Y(x_1), \dots, Y(x_K), X, Y)$ where $X=X(z)$ and $Y=Y(X)$. 

\begin{figure}[h]
\centering
\begin{tikzpicture}[scale=0.3,thick, node distance=5mm and 15mm] 
\node[name=a00]{$(X\!=\!1,Y\!=\!1)$};
\node[name=a01,below = of a00]{$(X\!=\!1,Y\!=\!2)$};
\node[name=a10,below = of a01]{$(X\!=\!2,Y\!=\!1)$};
\node[name=a11,below = of a10]{$(X\!=\!2,Y\!=\!2)$};
\node[name=b00, left= of a00,xshift=-2cm]{$(Y(x_1)\!=\!1,Y(x_2)\!=\!1)$};
\node[name=b01,below = of b00]{$(Y(x_1)\!=\!1,Y(x_2)\!=\!2)$};
\node[name=b10,below = of b01]{$(Y(x_1)\!=\!2,Y(x_2)\!=\!1)$};
\node[name=b11,below = of b10]{$(Y(x_1)\!=\!2,Y(x_2)\!=\!2)$};
\node[name=a,above = 2mm of a00]{${\cal B}$};
\node[name=b,above = 2mm of b00]{${\cal A}$};
\draw[-] (a00.west) -- (b00.east);
\draw[-] (a00.west) -- (b01.east);
\draw[-] (a01.west) -- (b10.east);
\draw[-] (a01.west) -- (b11.east);
\draw[-] (a10.west) -- (b00.east);
\draw[-] (a10.west) -- (b10.east);
\draw[-] (a11.west) -- (b01.east);
\draw[-] (a11.west) -- (b11.east);
\end{tikzpicture}
\caption{Illustration of pairs $(\bm{a},\bm{b}) \in {\cal A} \times {\cal B}$ when $K=M=2$ under a fixed instrument arm $z$. Each edge corresponds to a coherent pair.}
\label{fig:coherent}
\end{figure}

\vspace{-0.3cm}The coherence relation leads to neighbors in the sense of \cref{def:neighbor}. 
In the example of \cref{fig:coherent}, the assignments $(Y(x_1)=1, Y(x_2)=1), (Y(x_1)=1, Y(x_2)=2)\in {\cal A}$ are both neighbors of $(X=1, Y=1)\in {\cal B}$. In general, each of the $M^{K}$ elements in $\cal A$ is connected to $K$ neighbors in ${\cal B}$, while each of the $MK$ elements in ${\cal B}$ is connected to $M^{K-1}$ neighbors in $\cal A$. The total number of edges in the bipartite graph is $KM^K$. 

Recall that our goal is to show ${\cal T}\subseteq \phi({\calm}_1)$. That is, we need to show that given any pair
\[ \left(P(Y(x_1),\dots, Y(x_K)), \;P(X,Y\mid Z)\right)\in\cal T, \]
there exists a joint distribution $P(Z,X(z_1),\dots, X(z_Q), Y(x_1),\dots, Y(x_K))$ in ${\calm}_1$ such that
\[\phi(P(Z, X , Y(x_1),\dots,Y(x_K)))= \left(P(Y(x_1),\dots, Y(x_K)), \;P(X,Y\mid Z)\right),\]
where $X=X(Z)$. 

Our proof strategy breaks this problem down by considering each $Z$ arm in turn. Specifically, the next lemma shows that if for each $z \in [Q]$, $P(X,Y\mid Z=z)$ is compatible with $P(Y(x_1),\ldots, Y(x_K))$
in that there exists a compatible joint distribution $P(X(z),Y(x_1),\ldots, Y(x_K))$, then there exists a single joint distribution 
$P(Z, X(z_1),\dots, X(z_Q),\allowbreak Y(x_1),\ldots, Y(x_K))$
over all of the $X(z)$ and $Y(x)$ potential outcomes that is compatible with every $Z$ arm. 

\begin{lemma}\label{lemma:sepz}
    
    Given a set of $Q$ distributions $P_q(X(z_q), Y(x_1), \dots, Y(x_K))$ for $q\in [Q]$ that agree on the common marginal, i.e., $P_q(Y(x_1),\dots ,Y(x_K))=P_{q'}(Y(x_1),\dots ,Y(x_K))$ for all $q, q'\in [Q]$, 
then there exists a single joint distribution 
\[P(X(z_1),\dots, X(z_Q), Y(x_1),\dots, Y(x_K))\] 
that agrees with each of these $Q$ marginals.
    \end{lemma}
    \begin{proof}
    We may form a joint distribution
\[P^*(X(z_1),\dots,X(z_Q),Y(x_1),\dots,Y(x_K))=\frac{\prod_{q=1}^Q P_q (X(z_q),Y(x_1),\dots, Y(x_K))}{P_1(Y(x_1),\dots, Y(x_K))^{Q-1}}.\]
    The resulting distribution $P^*$ agrees with each $P_q$ on the $(X(z_q), Y (x_1),\dots, Y (x_K))$ margin.
    Though not important for our argument, we note that $P^*$ enforces the joint conditional independence of the $X(z)$ counterfactuals given $Y (x_1),\dots, Y (x_K)$.
\end{proof}

Then, for sufficiency, we need to prove ${\cal T}\subseteq \phi({\calm}_i)$ for $i=1,\dots,5$, where the map $\phi$ is given by \cref{eqs:phi}.
By \cref{lemma:models}, it suffices to just show ${\cal T}\subseteq \phi({\calm}_1)$. That is, we shall show that given any $\left(P'(Y(x_1),\dots, Y(x_K)), P(X,Y\mid Z)\right)\in\cal T$,
there exists a joint distribution 
\[ P^{\ast}(Z,X(z_1),\dots, X(z_Q), Y(x_1),\dots, Y(x_K)) \in \calm_1 \]
such that 
\[\phi \left(P^{\ast}(Z, X , Y(x_1),\dots,Y(x_K)) \right)=\left(P'(Y(x_1),\dots, Y(x_K)),\; P(X,Y\mid Z)\right). \] The details of the proof can be found in Supplemental~\cref{apx:proof-theo1-sufficiency}.

\vspace{-0.6cm}\section{Eliminating redundant inequalities} \label{sec:redundancy}\vspace{-0.3cm}
Our proof of \cref{theo:redundancy} is based on the following general result, which characterizes the extremal points of the inequalities given by \cref{thm:strassen}.

\begin{proposition} \label{prop:redundancy}
Consider the set of non-trivial inequalities given by \cref{thm:strassen} that characterize the existence of a coupling $\check{P}$ supported on $\mathcal{R}$: 
\begin{equation}
P_{\cal A}(U)\leq P_{\cal B}({\cal N}_{\cal R}(U)),  \quad \emptyset \subset U \subset {\cal A}. 
\end{equation}
\vspace{-0.3cm} For $\emptyset \subset U \subset {\cal A}$, define 
\[ \mathcal{R}(U):= \Big[\mathcal{R} \cap (U \times \mathcal{N}_{\mathcal{R}}(U))\Big] \; \cup \; \left[\mathcal{R} \cap (\overline{U} \times \overline{\mathcal{N}_{\mathcal{R}}(U)})\right].  \]
Then the inequality corresponding to $U$ is redundant\footnote{Given a set of inequalities describing a polytope, the goal is to find a subset of inequalities that are non-redundant and describe the same polytope. This can be achieved by simply removing every inequality that is redundant relative to the original set because we presume that no two inequalities in the `set' are identical.} if and only if there exists $U' \neq U$, $\emptyset \subset U' \subset \mathcal{A}$ such that $\mathcal{R}(U) \subseteq \mathcal{R}(U')$. 
\end{proposition}

The proof of \cref{prop:redundancy} is in Supplemental~\cref{apx:proof-prop1}.

\cref{theo:redundancy} can be established by verifying the condition in \cref{prop:redundancy} specific to the coherence relation $\mathcal{R}_C$ in the following three parts. Let $\calv=\calv^{(1)}\times\cdots\times\calv^{(K)} \subset {\cal A}$.
\begin{enumerate}[(I)]
    \item When there exists a single $k^{\ast}$ such that $|\calv^{(k^*)}|=M-1$ and $\calv^{(k)} = [M]$ for every $k \neq k^{\ast}$, we show the non-existence of ${\calv}' \neq \calv$, $\emptyset \subset \calv' \subset \mathcal{A}$ with $\mathcal{R}_{C}(\calv) \subseteq \mathcal{R}_{C}(\calv')$. 
    
    \item When $\calv^{(k)} \neq [M]$ and $\calv^{(k^{\ast})} \neq [M]$ for $k \neq k^{\ast}$, we also show the non-existence of ${\calv}' \neq \calv$, $\emptyset \subset \calv' \subset \mathcal{A}$ satisfying $\mathcal{R}_{C}(\calv) \subseteq \mathcal{R}_{C}(\calv')$. 
    
    \item Any other inequality in \eqref{theo:1cont}, corresponding to $\calv$ with $|\calv^{(k^\ast)}|< M-1$ for a single $k^\ast$ and $\calv^{(k)} = [M]$ for every $k \neq k^{\ast}$, must be redundant. To show this, we demonstrate a set ${\calv}' \neq \calv$, $\emptyset \subset \calv' \subset \mathcal{A}$ such that $\mathcal{R}_{C}(\calv) \subseteq \mathcal{R}_{C}(\calv')$.
\end{enumerate}
The details are deferred to \cref{apx:proof-theo-redun}. It is worth mentioning that although \cref{prop:redundancy} is stated for a larger set of inequalities (i.e., corresponding to all non-trivial $U$ instead of just Cartesian-form $U$) than \eqref{theo:1cont}, the result still applies because an inequality's redundancy is determined relative to the \emph{polytope} defined by a set of inequalities. 

For a fixed, single instrument arm, the proof of \cref{prop:redundancy} shows that the extreme points (i.e., vertices) of the polytope exactly correspond to the edges defined by $\mathcal{R}_C$. As can be seen from \cref{fig:coherent}, every edge pairs a principal stratum \citep{identifictionIV1996,frangakis2002principal} of the population (e.g., ``always recover'' $Y(x_1) = Y(x_2) = 2$) with a compatible observed value (e.g., $X=1,Y=2$). We discuss this in more detail in Supplemental Section~\ref{sec:cdd}. 

\cref{prop:redundancy} can be extended to $Q \geq 1$ instrument arms. The object to characterize is the counterfactual distribution $P'(Y(x_1), \dots, Y(x_K))$ along with the observed distributions $P(X,Y \mid Z=z)$ for $z \in [Q]$, which together are identified as a polytope that is a subset of the product space 
\begin{equation} \label{eqs:prod-space}
\Delta^{M^K - 1} \times \left(\Delta^{KM-1} \right)^{Q}.
\end{equation}
By an argument similar to the proof of \cref{prop:redundancy}, each extreme point of this polytope corresponds to a principal stratum of the population (e.g., $Y(x_1)=1, Y(x_2) = 2$ when $K=2$) and a compatible observed value in each instrument arm (e.g., $X=1, Y=1$ for $z=1$, $X=2, Y=2$ for $z=2$, etc.). 

\begin{corollary} \label{cor:V-rep}
Under each of the models ${\calm}_1, \dots, {\calm}_5$, the joint counterfactual distribution $P(Y(x_1), \dots, Y(x_K))$ and the observed distributions $\left(P(X,Y \mid Z=z): z \in [Q]\right)$ are characterized as a polytope in the product space \eqref{eqs:prod-space}. The polytope has $M^K K^Q$ extreme points, given by
\[ \left\{ \delta_{y(1), \dots, y(K)} \times \prod_{z \in [Q]} \delta_{x_{z},y(x_z)}: \; y(1), \dots, y(K) \in [M],\; x_{1}, \dots, x_{Q} \in [K] \right\}.\]
\end{corollary}

Here the term 
$\delta_{y(1), \dots, y(K)}$
determines a single principal stratum of the outcome, while each term $\delta_{x_{z},y(x_z)}$ specifies a compatible degenerate observed distribution.

\begin{proof}
This follows from \cref{theorem:main-result} and the proof of \cref{prop:redundancy}. 
\end{proof}

\begin{remark}[Complexity]
In \cref{sec:inference}, we will describe a convex programming approach that streamlines partial identification and statistical inference, which treats the counterfactual and observed distributions as unknowns in the program. To express the unknowns in terms of a convex combination of the extreme points above, the V-representation \citep[p.~29]{polytope-book} approach requires $\bigO(M^K K^Q)$ parameters with $\bigO(M^K K^Q)$ inequalities. In contrast, the H-representation based on \cref{theo:redundancy} requires $\bigO(M^K + QKM)$ parameters and $\bigO(Q \,2^{MK})$ inequalities (dominated by $r$ in \eqref{eqs:numbounds}). Both representations overcome the super-exponential complexity from directly applying Artstein's inequality. Compared to the V-representation, the H-representation has the advantage of a linear dependence on $Q$, making it more suitable for optimization in most cases. However, in the setting where $Q$ and $K$ are small but $M$ is large, the V-representation can be preferable.
\end{remark}

\vspace{-0.6cm}\section{Statistical inference on partial identification bounds}\label{sec:inference}\vspace{-0.3cm}

Given an observed distribution, the inequalities of \cref{theorem:main-result,theo:redundancy}, which define the set of compatible counterfactual distributions, can be used to obtain partial identification bounds on any linear functional of the joint counterfactual distribution, such as a marginal probability $P(Y(x_i)=y)$ or an ATE between two treatment levels, with the help of existing linear programming software. To account for the sampling variability of the empirical distribution, in this section, we show how to construct finite-sample confidence intervals for such functionals that are guaranteed to contain the true values with probability no less than a pre-specified level. The construction is based on a concentration inequality introduced by \citet{GuoChernoff}, which provides a tail bound for the Kullback--Leibler divergence between the empirical distribution and the true distribution under multinomial sampling.
For a specified level $\alpha$, the bound asserts that there exists a threshold $t_{\alpha}$, depending on $K$, $M$, and the sample size in each arm, such that with probability at least $1-\alpha$, we have
\begin{equation} \label{eqs:conf-region}
 \sum_{z=1}^Q n_z \KL \left(\hat{P}(X,Y \mid Z=z) \| P(X, Y \mid Z=z) \right) \leq t_{\alpha}, 
\end{equation}
where $\hat{P}$ denotes the empirical distribution, $\KL$ denotes the Kullback--Leibler divergence and $n_z$ is the sample size in the instrument arm $z$. Due to the convexity of $\KL(\hat{P} \| \cdot)$, the bound above induces a convex confidence region for the collection of observed distributions $\left(P(X, Y \mid Z=z): z=1,\dots,Q \right)$, centered around their empirical counterparts. Further, by combining \eqref{eqs:conf-region} with \cref{theorem:main-result}, we obtain a conservative $(1-\alpha)$-level convex confidence region for the counterfactual probabilities $P'\left(Y(x_1), \dots, Y(x_K)\right)$.
By minimizing and maximizing any linear functional of the counterfactual distribution, such as an ATE, we hence obtain a confidence interval that is guaranteed to contain the true value with probability at least $1-\alpha$.
The procedure can be applied to a number of different linear functionals simultaneously and with probability at least $(1-\alpha)$ all intervals will contain their respective estimands --- this family-wise coverage guarantee follows because these intervals are all based on the same confidence region for the observed distribution.
{Our inference procedure guarantees finite-sample, simultaneous coverage without requiring any extra assumptions.
Moreover, as we will see next, the constructed confidence intervals shrink to the identified intervals at the $n^{-1/2}$ rate.
To our knowledge, no method of this kind exists in the literature on this topic. 
In fact, regular inference in this case is impossible without assuming extra regularity and margin conditions \citep{hirano2012impossibility,fang2019inference,voronin2025linear}.}


We now describe the inference algorithm in more detail. For $z=1,\dots,Q$, we use $p_z$ and $\hat{p}_z$ to denote $P(X,Y \mid Z=z)$ and $\hat{P}(X,Y \mid Z=z)$ respectively, both of which are vectors in the probability simplex $\Delta^{KM-1}$. The critical value $t_{\alpha}$ in \eqref{eqs:conf-region} can be determined numerically from the Chernoff bound 
\begin{equation} \label{eqs:tail-bound}
 P\left( \sum_{z=1}^Q n_z \KL(\hat{p}_z \| p_z) > t \right) \leq \min_{\lambda \in [0,1]} \exp(-\lambda t) \prod_{z=1}^Q G_{KM, n_z}(\lambda), 
\end{equation}
as the value of $t$ at which the right-hand side equals $\alpha$; this is implemented in the R package \texttt{multChernoff} available on CRAN. In the equation above,
\[ G_{KM, n_z}(\lambda) := \sum_{m=0}^{n_z} \frac{n_z!}{n_z^m\,(n_z-m)!} {m+KM-2\choose KM-2}\lambda^m \]
is a polynomial that upper bounds the moment generating function of $n_z \KL(\hat{p}_z \| p_z)$ \citep[Theorem 1]{GuoChernoff}; the bound \eqref{eqs:tail-bound} then follows from a standard Chernoff bound argument by independence of data across $Z$ arms. 
Under the classical asymptotic regime with fixed $K,M$ and $n_z \to \infty$, it can be shown that $G_{KM, n_z}(\lambda)$ converges to a certain moment generating function,\footnote{It converges to the MGF of $\mathrm{Gamma}(KM-1, 1)$, which has an extra factor of 2 in the shape parameter, compared to the large-sample distribution $\mathrm{Gamma}((KM-1)/2, 1)$.} which implies the $n^{-1/2}$ rate of the resulting intervals. 

\cref{theorem:main-result,theo:redundancy} yield a set of $r$ non-redundant inequalities that characterize the set of joint counterfactual distributions, where $r$ is given by \eqref{eqs:numbounds}.
In the algorithm presented in Supplemental~\cref{supp:algo-1}, we use binary matrices $H'\in\{0,1\}^{r\times M^K}$ and $H\in\{0,1\}^{r\times KM}$ to encode these inequalities, with one inequality per row. Each row of $H'$ indicates which joint counterfactual outcomes $(Y(x_1),\dots, Y(x_K))$ are in the Cartesian product $\calv^{(1)}\times\cdots\times\calv^{(K)}$, and each row of $H$ represents which observed probabilities $P(X=i, Y=m\mid Z=z)$ contribute to the right-hand side $\sum_{i=1}^KP(X=i, Y\in\calv^{(i)}\mid Z=z)$. Together, the inequalities are encoded as $H' p'\leq H p_z$ for every $z \in [Q]$; matrices $H',H$ can be obtained by the method described in Supplement~\ref{sec:cdd}. Given a collection of linear functionals of the counterfactual distribution, \cref{alg:convex} presents a convex program for constructing the confidence intervals. 
The next theorem states the algorithm's statistical guarantee. 

\begin{theorem} \label{thm:coverage}
Suppose data is generated from an IV model in the sense of any $\calm_i$ in \cref{tab:models}. Let $P_0(Y(x_1),\dots, Y(x_K))$ be the underlying counterfactual distribution. For each $j \in [J]$, let $f_j$ be a linear functional of the counterfactual distribution and let $[l_j, u_j]$ be the corresponding confidence interval obtained from \cref{alg:convex} (in Supplement~\ref{supp:algo-1}). Then, with probability at least $1-\alpha$, it holds that $l_j \leq u_j$ and $f_j(P_0) \in [l_j, u_j]$ simultaneously for all $j \in [J]$.
\end{theorem}
\begin{proof}
The tail bound \eqref{eqs:tail-bound} guarantees that with probability at least $1-\alpha$, the feasible region for $(p_z: z \in [Q])$ of the convex program contains the true population distribution. Then, by \cref{theorem:main-result}, it follows that with probability at least $1-\alpha$, the feasible region for $p'$ contains $P_0$, which implies that $l_j \leq u_j$ and $f_j(P_0) \in [l_j, u_j]$ for every $j \in [J]$. 
\end{proof}

  
  

Note that, when the IV model is not assumed a priori and \cref{alg:convex} returns $l_j=+\infty$ and $u_j=-\infty$, the IV model is falsified by the observed data.

\vspace{-0.6cm}\section{Motivating example revisited} \label{sec:real-data}\vspace{-0.3cm}

We now revisit the Minneapolis Domestic Violence Experiment introduced in \cref{sec:motivating}. Using four researchers, we compare the results obtained by our systematic approach to those obtained by two {\em ad hoc}, Procrustean approaches that attempt to apply existing methods for binary treatment $X$ to the dataset. Consider the following three pairwise average treatment effects: \vspace{-0.5cm}
\begin{align*}
\text{ATE}_j &:= P(Y(x=x_j)=2) - P(Y(x=x'_j)=2), \quad j=1,2,3, \\
&= P(\text{re-offence in 6 months under treatment $x_j$}) \\
& \qquad - P(\text{re-offence in 6 months under treatment $x'_j$}),
\end{align*}
\vspace{-0.3cm}for (1) {\em Advise} ($x_1=\text{Adv}$) vs. {\em Arrest} ($x_1'=\text{Arr}$), (2) {\em Separate} ($x_2=\text{Sep}$) vs. {\em Arrest} ($x_2'=\text{Arr}$), and (3) {\em Separate} ($x_3=\text{Sep}$) vs. {\em Advise} ($x_3'=\text{Adv}$).

\begin{sidewaystable}[p]
\caption{Results for the Minneapolis Domestic Violence Experiment obtained by different researchers}\label{tab:analysis}
\centering
\small
\begin{tabular}{cccccccc} 
\toprule
& \multirow{2}{*}{\textbf{}}  & \multicolumn{2}{c}{\textbf{Advise vs. Arrest}}     & \multicolumn{2}{c}{\textbf{Separate vs. Arrest}}   & \multicolumn{2}{c}{\textbf{Separate vs. Advise}} \\ \hline  Researcher &   & \textbf{Plug-in} & \textbf{CI (95\%)} & \textbf{Plug-in} & \textbf{CI (95\%)} & \textbf{Plug-in} & \textbf{CI (95\%)} \\ \hline
1  & \textbf{All data}   & (0.019, 0.252)   & (-0.374, 0.633)   & (0.057, 0.343)   & (-0.346, 0.702)   & (-0.184, 0.312)  & (-0.583, 0.683)  \\ \hline
2   & \textbf{\begin{tabular}[c]{@{}c@{}}Delete $X\!=\!\text{Sep}$,\\ $X\!=\!\text{Adv}$, \\  or $X\!=\!\text{Arr}$ \end{tabular}}  & NA               & (-0.283, 0.526)  & NA               & (-0.264, 0.625)   & NA    & (-0.355, 0.457)   \\ \hline
3    & \textbf{\begin{tabular}[c]{@{}c@{}}Binary IV model: \\ Delete $X\!=\!\text{Sep}$ and $Z\!=\!\text{Sep}$,\\ $X\!=\!\text{Adv}$ and $Z\!=\!\text{Adv}$,\\  or $X\!=\!\text{Arr}$ and $Z\!=\!\text{Arr}$\end{tabular}} & (0.037, 0.214)   & (-0.241, 0.506)                   & (0.066, 0.317)   & (-0.222, 0.598)                   & (-0.003, 0.121)  & (-0.337, 0.446)        \\ \hline
\multirow{3}{*}{4} & \textbf{Delete $Z\!=\!\text{Arr}$}  & (-0.675, 0.317)  & (-0.885, 0.621)     & (-0.637, 0.407)  & (-0.858, 0.691)  & (-0.184, 0.312)  & (-0.533, 0.639) \\ & \textbf{Delete $Z\!=\!\text{Adv}$}   & (-0.111, 0.856)  & (-0.407, 0.981)   & (0.057, 0.343)   & (-0.285, 0.660)   & (-0.788, 0.442)  & (-0.953, 0.721) \\ & \textbf{Delete $Z\!=\!\text{Sep}$}  & (0.019, 0.252)   & (-0.314, 0.586)  & (-0.092, 0.864)  & (-0.394, 0.985)  & 
(-0.333, 0.833) & (-0.628, 0.973)  \\ \bottomrule
\end{tabular}
\end{sidewaystable}

\begin{description}
\item [Researcher 1] They used all the data for all three pairwise ATEs (our approach). \vspace{-0.2cm} 

\item [Researcher 2] To make $X$ binary, they omitted participants who took the treatment that was not of interest for a given pairwise ATE. For example, when estimating the ATE comparing {\em Arrest} vs. {\em Advise}, they discarded the data from the treatment arm $X=\text{Sep}$. \vspace{-0.2cm}

\item [Researcher 3] Going beyond Researcher 2, they also omitted the instrument arm that assigns the treatment not of interest for a pairwise ATE. That is, when estimating the ATE comparing {\em Arrest} vs. {\em Advise}, they discarded the data with $X=\text{Sep}$ or $Z=\text{Sep}$.\vspace{-0.2cm}

\item [Researcher 4] They discarded one instrument arm to make $Z$ binary and then applied our approach with ternary $X$ and binary $Y$.\vspace{-0.2cm}
\end{description}

\cref{tab:analysis} shows the results obtained by the four researchers: each researcher computed the plug-in estimate (ignoring sampling variability) for the partially identified bounds on each ATE, and also constructed a 95\% confidence interval for each ATE using \cref{alg:convex}; the symbol NA indicates that the set of compatible counterfactual distributions is empty. Computation was performed using the \texttt{CVXR} package \citep{cvxr2020} with the \texttt{ECOS} solver \citep{domahidi2013ecos} on an ARM64 personal computer. The resulting run time, reported in Supplemental \cref{tab:analysis_runtime}, had a maximum of 5.6 seconds, demonstrating computational feasibility and efficiency. 

The bounds obtained by Researcher 1 for Advise vs. Arrest and Separate vs. Arrest are positive, indicating higher re-offence rates following non-arrest responses. 
These results are consistent with both the original findings of \cite{berkandsherman} and those of \cite{angristcrime}: namely, that 
in the Minneapolis Domestic Violence Experiment the
Arrest strategy was most effective in deterring re-offending. As explained in \cref{sec:motivating}, the analyses carried out by Researchers 2 and 3 are biased due to selection on $X$, which violates the independence assumption and hence renders the imposed binary IV model invalid (see \cref{fig:dag_assump}). In fact, the plug-in estimates from Researcher 2 fall outside the IV model. 
 
In addition to Researcher 1's analysis, that of Researcher 4 is also valid: discarding an instrument arm does not introduce bias because the instrument is randomized. The plug-in estimates obtained by Researcher 4, when removing the ``less relevant'' $Z$ arm, are numerically equal to those obtained by Researcher 1.
However, in general, using data from all the instrument arms
will lead to plug-in intervals that are no wider and sometimes strictly tighter than those obtained by Researcher 4. In our example, the
confidence intervals obtained by Researcher 4, when removing the less relevant arm, are narrower than those obtained by Researcher 1. However, it is important to remember that those obtained by Researcher 1 have simultaneous coverage, while those obtained by Researcher 4 only guarantee marginal coverage. 

\vspace{-0.6cm}\section{Conclusion and discussion}\label{sec:discussion}\vspace{-0.3cm}

In this paper, we provide a set of linear inequalities that describe the relationship between the joint counterfactual distribution and the observed data distribution under categorical IV models, where the instrument, treatment, and outcome all take finitely many values. The set of inequalities are shown to be necessary, sufficient and non-redundant under the various IV models considered in the literature. 
{Our work fills an important gap between existing general characterizations that involve a super-exponential number of redundant inequalities and simpler formulations that provide only necessary conditions and therefore do not yield sharp bounds or sharp falsification criteria. By deriving a minimal sharp characterization, our results make valid and computationally tractable analysis possible for IV models with categorical instrument, treatment and outcome levels.}
Our results are established using a version of Strassen's theorem on finite sets (\cref{thm:strassen,prop:redundancy}), which may be of interest for other problems. Further, we demonstrate how to construct confidence intervals for ATEs through a convex program that incorporates the IV inequalities along with a finite-sample bound that handles sampling variability. 

We leave the following items for future work: (1) extending the results to a continuous outcome and/or instrument, (2) obtaining explicit instrumental inequalities (\cref{rem:falsify}) for a falsification test, and (3) constructing less conservative confidence intervals. 


\spacingset{1}
\bibliography{reference}

\newpage
\spacingset{1.8}
\setsupplementnumbering

\begin{center}
  {\Large\bfseries
  Supplementary Materials for ``The Categorical Instrumental Variable Model''}
\end{center}

\bigskip

This supplement is organized as follows. \Cref{apx:proof-theo1,apx:proof-theo1-sufficiency} prove the necessity and sufficiency of the inequalities in \cref{theorem:main-result}, and \cref{apx:proof-prop1,apx:proof-theo-redun} prove \cref{prop:redundancy,theo:redundancy}. \Cref{supp:algo-1} gives the convex program for constructing the confidence intervals, and \cref{sec:cdd} derives the V- and H-representations of the model polytope. \Cref{sec:runtime} reports run times for the data analysis. \cref{sec:feasibility} constructs the falsification test for the IV model. 
Finally, \cref{appendix:russell,supp:KM} relate our characterization to the work of \citet{Russell} and \citet{KedagniMourifie}.



\section{Proof of necessity for \cref{theorem:main-result}} \label{apx:proof-theo1}

In light of \cref{lemma:models}, to establish the necessity of inequalities (\ref{theo:1cont}), it suffices to show (i) $\phi({\calm}_3)\subseteq {\cal T}$, (ii) $\phi({\calm}_4)\subseteq {\cal T}$, and (iii) $\phi({\calm}_5)\subseteq {\cal T}$. Recall that $Y(x) := Y(x,Z)$ is the potential outcome of $Y$ when only $X$ is intervened on, which is essential to the proof in this section. 

\subsection{Proof of necessity under model $\calm_3$}
In \cref{assumption:indep}, we considered various versions of independence between the instrument $Z$ and the potential outcome $Y(x,z)$. In fact, combining the independence assumption with an appropriate exclusion restriction leads to the independence between $Z$ and $Y(x)$, as demonstrated by the next result for $\calm_3$.

\begin{lemma}\label{lem:joint-ind}
Under $\calm_3$, we have $Z \ind Y(x_1),\ldots ,Y(x_K)$.
\end{lemma} 

\begin{proof} For any $z,\tilde{z} \in [Q]$ and any $y^1, \dots, y^K \in [M]$, we have
\begin{align*}
\MoveEqLeft[10]{P\left(\left.Y(x_1)=y^1,\ldots ,Y(x_K)=y^K\right| Z=z\right)}\\
\text{(consistency)}\; &= P\left(\left.Y(x_1,z)=y^1,\ldots ,Y(x_K,z)=y^K\right| Z=z\right)\\
\text{(by joint independence \eqref{eq:joint-indep})} &= P\left(Y(x_1,z)=y^1,\ldots ,Y(x_K,z)=y^K\right)\\
\text{(by joint stochastic exclusion \eqref{eq:joint-stoc-exclusion})} &= P\left(Y(x_1,\tilde{z})=y^1,\ldots ,Y(x_K,\tilde{z})=y^K\right)\\
\text{(by joint independence \eqref{eq:joint-indep})} &= P\left(\left.Y(x_1,\tilde{z})=y^1,\ldots ,Y(x_K,\tilde{z})=y^K \right| Z=\tilde{z}\right)\\
\text{(consistency)}\; &= P\left(\left.Y(x_1)=y^1,\ldots ,Y(x_K)=y^K \right| Z=\tilde{z}\right).
\end{align*}
\end{proof}


\begin{proof}[Proof of $\phi({\calm}_3)\subseteq {\cal T}$] For any $P \in \calm_3$, it holds that 
     \begin{align*}
\MoveEqLeft[6]{\sum_{i=1}^K P\left(\left.X=i, Y\in \calv^{(i)}\right| Z=z\right)}\\[-4pt]
 \text{(by consistency)\quad}   &=  \sum_{i=1}^K P\left(\left. X=i, Y(x_i)\in \calv^{(i)}\right| Z=z \right)\\
    & \geq \sum_{i=1}^K 
   P\left(\left. X=i, Y(x_1)\in \calv^{(1)},\ldots , Y(x_K)\in \calv^{(K)}\right| Z=z\right)\\
    &= P\left(\left. Y(x_1)\in \calv^{(1)},\ldots , Y(x_K)\in \calv^{(K)}\right| Z=z \right)\\
\text{(by \cref{lem:joint-ind})\quad}    &= P\left(Y(x_1)\in \calv^{(1)},\ldots , Y(x_K)\in \calv^{(K)}\right).
\end{align*}
\end{proof}


\subsection{Proof of necessity under the SWIG model $\calm_4$}
\begin{proof}[Proof of $\phi({\calm}_4)\subseteq {\cal T}$]
Under $\calm_4$, we have $Y(x) := Y(x,Z) = Y(x,z)$ almost surely by individual-level exclusion. Hence, the single-world independence then implies 
\begin{equation} \label{eqs:swig-ind}
Z \ind X(z), Y(x), \quad x \in [K], z \in [Q]. 
\end{equation}
For every $z \in [Q]$, we have
\begin{align*}
\MoveEqLeft[5]{\sum_{i=1}^K P\left(\left.X=i, Y\in \calv^{(i)}\right| Z=z\right)}  \\[-4pt]
\text{(consistency)}\, &= \sum_{i=1}^K P\left(\left. X(z)=i, Y(x_i)\in \calv^{(i)}\right| Z=z\right)\\
\text{(by \cref{eqs:swig-ind})}\,    &= \sum_{i=1}^K P\left(X(z)=i, Y(x_i)\in \calv^{(i)}\right)\\
    &\geq \sum_{i=1}^K P\left(X(z)=i, Y(x_1)\in \calv^{(1)},\ldots, Y(x_i)\in \calv^{(i)},\ldots, Y(x_K)\in \calv^{(K)}\right)\\
    &= P\left(Y(x_1)\in \calv^{(1)},\ldots , Y(x_i)\in \calv^{(i)},\dots, Y(x_K)\in \calv^{(K)}\right).
\end{align*}
\end{proof}

\subsection{Proof of necessity under the latent model $\calm_5$}
We first prove the following lemma. 

\begin{lemma}\label{lemma:latent-ind}
Under the latent model $\calm_5$, we have 
\[
Y(x) \ind X, Z \mid U.
\]
\end{lemma} 
\begin{proof} For any $x^{\ast}, x^{\ast\ast} \in [K]$, any $z^{\ast}, z^{\ast\ast} \in [Q]$ and any value $u$ of $U$, we have
\begin{align*}
\MoveEqLeft[10]{P\left(\left.Y(x)=y \;\right| X=x^*, Z=z^*, U=u\right)}\\
\text{(consistency)}\; &= P(Y(x,z^*)=y \mid X=x^*, Z=z^*, U=u)\\
\text{(by \cref{eq:latent-exogeneity})}\; &= P(Y(x,z^*)=y \mid U=u)\\
\text{(by \cref{eq:latent-exclusion})}\; &= P(Y(x,z^{**})=y \mid U=u)\\
\text{(by \cref{eq:latent-exogeneity})}\; &= P(Y(x,z^{**})=y \mid X=x^{**}, Z=z^{**}, U=u)\\
\text{(consistency)}\; &= P(Y(x)=y \mid X=x^{**}, Z=z^{**}, U=u).
\end{align*}
\end{proof}

\begin{proof}[Proof of $\phi({\calm}_5)\subseteq {\cal T}$] Without much loss of generality, we assume $U$ is a discrete random variable in the proof below. We have
\begin{align*} 
\MoveEqLeft{\sum_{i=1}^K P\left(\left. X=i, Y\in \calv^{(i)} \;\right| Z=z\right)}\\[-4pt]
    &=\sum_u \sum_{i=1}^K P\left(\left. X=i, Y(x_i) \in \calv^{(i)}, U=u\;\right| Z=z\right) \\
    &\overset{(a)}{=} \sum_u \left(\sum_{i=1}^K P\left(\left. Y(x_i)\in \calv^{(i)}\right| X=i, U=u, Z=z\right)\cdot P\left(\left. X=i\;\right| U=u, Z=z\right)\right)\\[-8pt]
    &\kern120pt\quad\quad\quad\quad \ \ \ \cdot P(U=u\mid Z=z)\\
    &\overset{(b)}{=} \sum_u\left(\sum_{i=1}^K P\left(\left.Y(x_i)\in \calv^{(i)}\right| U=u\right)\cdot P(X=i\mid U=u, Z=z)\right)\cdot P(U=u)\\
    &\geq \sum_u \left(\sum_{i=1}^K P\left(\left. Y(x_1)\in \calv^{(1)},\dots, Y(x_i)\in \calv^{(i)},\dots, Y(x_K)\in\calv^{(K)}\right| U=u\right)\right. \\[-10pt]
    &\kern140pt\ \ \ \left.
    \vphantom{\sum_{i=1}^K}
    \cdot P(X=i\mid U=u, Z=z)\right)\cdot P(U=u)\\[-4pt]
    &\geq \sum_u \left(
    \vphantom{\sum_{i=1}^K}
    P\left(\left. Y(x_1)\in \calv^{(1)},\dots, Y(x_i)\in \calv^{(i)},\dots, Y(x_K)\in\calv^{(K)}\right| U=u\right)\right. \\[-14pt]
    &\kern140pt\ \ \ \left.\cdot \sum_{i=1}^K P(X=i\mid U=u, Z=z)\right)\cdot P(U=u)\\
    &= \sum_u P\left(\left.Y(x_1)\in\calv^{(1)},\dots, Y(x_K)\in\calv^{(K)}\right| U=u\right)\cdot P(U=u)\\
    &= P\left(Y(x_1)\in\calv^{(1)},\dots, Y(x_K)\in\calv^{(K)}\right),
\end{align*}
where step (a) uses consistency, and step (b) uses \cref{eq:latent-exogeneity,lemma:latent-ind}.
\end{proof}

\section{Proof of sufficiency for \cref{theorem:main-result}}\label{apx:proof-theo1-sufficiency}

\begin{proof}[Proof of sufficiency of the inequalities in \cref{theorem:main-result}]
For sufficiency, we need to prove ${\cal T}\subseteq \phi({\calm}_i)$ for $i=1,\dots,5$, where the map $\phi$ is given by \cref{eqs:phi}.
By \cref{lemma:models}, it suffices to just show ${\cal T}\subseteq \phi({\calm}_1)$. That is, we shall show that given any $\left(P'(Y(x_1),\dots, Y(x_K)), P(X,Y\mid Z)\right)\in\cal T$,
there exists a joint distribution 
\[ P^{\ast}(Z,X(z_1),\dots, X(z_Q), Y(x_1),\dots, Y(x_K)) \in \calm_1 \]
such that 
\[\phi \left(P^{\ast}(Z, X , Y(x_1),\dots,Y(x_K)) \right)=\left(P'(Y(x_1),\dots, Y(x_K)),\; P(X,Y\mid Z)\right). \]

Under model ${\calm}_1$, for $z \in [Q]$, we have 
\begin{equation} \label{eq:consistM1}
P(X(z)=i, Y(x_i)=j)=P(X(z)=i, Y(x_i)=j\mid Z=z)=P(X=i, Y=j\mid Z=z).
\end{equation}
    
\cref{lemma:sepz} implies that we can consider each level $z \in [Q]$ of $Z$ separately: if we can construct $Q$ coupling distributions over $(X(z), Y(x_1),\dots, Y(x_K))$ for $z \in [Q]$ that each obey (\ref{eq:consistM1}) and agree on the $(Y(x_1),\dots, Y(x_K))$ margin, then we can form a single joint distribution. Hence, it remains to show that given any pair $(P'(Y(x_1),\dots, Y(x_K)),P(X,Y\mid Z))$ that satisfies the inequalities \eqref{theo:1cont}, there exist joint distributions $P_{z}(X(z),Y(x_1),\dots, Y(x_K))$ for $z \in [Q]$ such that
\begin{equation} \label{eqs:couple-z}
\begin{split}
    P_z\left(Y(x_1),\dots,Y(x_K)\right)&=_{d} P'\left(Y(x_1),\dots,Y(x_K)\right), \text{ and }\\
    P_z(X(z)=i, Y(x_i)=j) &= P(X=i,Y=j\mid Z=z) \text{ for all } i,j.
\end{split}
\end{equation}

We are ready to apply \cref{thm:strassen}. Let $z$ be fixed. Note that $P'(Y(x_1),\dots, Y(x_K))$ is a probability measure on $\mathcal{A} = [M]^K$ and $P(X,Y\mid Z=z)$ is a probability measure on $\mathcal{B} = [K] \times [M]$. We shall show that the inequalities \eqref{theo:1cont} suffice to ensure the existence of a desired joint distribution $P_z(X(z),Y(x_1),\dots, Y(x_K))$ that meets \cref{eqs:couple-z}. Inequalities (\ref{theo:1cont}) (modulo trivial inequalities) assert that for every $\calv^{(1)},\dots,\calv^{(K)} \subseteq [M]$, it holds that 
\[P'\left(Y(x_1) \in {\calv}^{(1)},\ldots,Y(x_K)\in {\calv}^{(K)}\right)\;\leq\; \sum_{i=1}^K P\left(\left.X\!=\!i, Y\in {\calv}^{(i)}\right| Z\!=\!z\right). \] 
We now compare them to the characterization in \cref{thm:strassen}. For any non-empty $\mathcal{U} \subseteq \mathcal{A} = [M]^K$, let $\mathcal{U}^{(1)}, \dots, \mathcal{U}^{(K)} \subseteq [M]$ be its coordinate-wise projections, which are also non-empty. By the coherence relation $\mathcal{R}_{{C}}$ defined in \cref{eqs:coherence}, the set of neighbors of $\mathcal{U}$ is given by
\[ \mathcal{N}_{\mathcal{R}_{{C}}}(\mathcal{U}) = \bigcup_{i=1}^K \{i\} \times \mathcal{U}^{(i)}. \]
Hence, \cref{thm:strassen} posits that for every non-empty $\mathcal{U} \subseteq [M]^K$, 
\begin{equation} \label{eqs:ineq-strassen}
P'\left(\left(Y(x_1),\dots,Y(x_K)\right) \in \mathcal{U} \right)\;\leq\; \sum_{i=1}^K P\left(\left.X\!=\!i, Y\in {\mathcal{U}}^{(i)}\right| Z\!=\!z\right).
\end{equation}
Yet, observe that it suffices to only consider every Cartesian-form $\mathcal{U}$, i.e., one satisfying $\mathcal{U} = \mathcal{U}^{(1)} \times \dots \times \mathcal{U}^{(K)}$, because among the sets with the same coordinate-wise projections (and hence the same RHS), this $\mathcal{U}$ maximizes the LHS. Collecting \cref{eqs:ineq-strassen} for non-empty Cartesian-form $\mathcal{U}$'s gives the inequalities (\ref{theo:1cont}).
\end{proof}

\section{Proof of \cref{prop:redundancy}}\label{apx:proof-prop1}
\begin{proof}
By the representation theorem for polytopes \citep[Theorem 2.15]{polytope-book}, an inequality is non-redundant iff the hyperplane $P_{\mathcal{A}}(U) = P_{\mathcal{B}}({\cal N}_{\cal R}(U))$ defines a facet of the polytope of pairs of marginal distributions $(P_{\cal A}, P_{\cal B})$ that are compatible with a coupling supported on $\cal R$. A facet is a face of the polytope that is bounded by a maximal (by inclusion) set of extremal points (i.e., vertices) of the polytope. Hence, it suffices to prove that $\mathcal{R}(U)$ (or more precisely, the corresponding pairs of point masses $\left\{(\delta_{\bm{a}}, \delta_{\bm{b}}): (\bm{a},\bm{b}) \in \mathcal{R}(U) \right\}$) is the set of extremal points on the face defined by $P_{\mathcal{A}}(U) = P_{\mathcal{B}}({\cal N}_{\cal R}(U))$. 

First, we show that for every $(\bm{a},\bm{b}) \in \mathcal{R}(U)$ the pair of distributions $(\delta_{\bm{a}}, \delta_{\bm{b}})$ forms an extremal point that satisfies
 $P_{\mathcal{A}}(U) = P_{\mathcal{B}}({\cal N}_{\cal R}(U))$.
 Consider the corresponding coupling $\check{P}=\delta_{(\bm{a},\bm{b})}$. Under $\check{P}$, for $(\bm{a}, \bm{b}) \in \mathcal{R} \cap (U \times \mathcal{N}_{\mathcal{R}}(U))$, on the implied margins we have $P_{\cal A}(U)= P_{\cal A}(\{\bm{a}\}) = P_{\cal B}({\cal N}_{\cal R}(U))= P_{\cal B}(\{\bm{b}\}) = 1$; similarly, for $(\bm{a}, \bm{b}) \in \mathcal{R} \cap (\overline{U} \times \overline{\mathcal{N}_{\mathcal{R}}(U)})$, we have $P_{\cal A}(U)= 1-P_{\cal A}(\{\bm{a}\}) = P_{\cal B}({\cal N}_{\cal R}(U))=1-P_{\cal B}(\{\bm{b}\})=0$. Hence in both cases we have
$P_{\mathcal{A}}(U) = P_{\mathcal{B}}({\cal N}_{\cal R}(U))$, so this equality defines a face. Furthermore, $(\delta_{\bm{a}}, \delta_{\bm{b}})$ is an extremal point because both $P_{\mathcal{A}}$ and $P_{\mathcal{B}}$ take the form of a point mass. 

Now we argue that the face defined by $P_{\mathcal{A}}(U) = P_{\mathcal{B}}({\cal N}_{\cal R}(U))$ cannot contain any other extremal point
besides those in ${\cal R}(U)$. To prove by contradiction, suppose there is an extremal point
that does not correspond to any 
$(\delta_{\bm{a}}, \delta_{\bm{b}})$ for
$(\bm{a}, \bm{b}) \in \mathcal{R}(U)$.
Let $\check{P}$ be any corresponding coupling measure.
Recall that $\check{P}(\mathcal{R})=1$ and we can decompose $\mathcal{R}$ as
\begin{align*}
\mathcal{R} &= \mathcal{R}(U) \,\cup\, \Big[\mathcal{R} \cap (\overline{U} \times \mathcal{N}_{\mathcal{R}}(U))\Big] \,\cup\, \Big[\mathcal{R} \cap (U \times \overline{\mathcal{N}_{\mathcal{R}}(U)})\Big]  \\
&= \mathcal{R}(U) \,\cup\, \Big[\mathcal{R} \cap (\overline{U} \times \mathcal{N}_{\mathcal{R}}(U))\Big],
\end{align*}
since $\mathcal{R} \cap (U \times \overline{\mathcal{N}_{\mathcal{R}}(U)}) = \emptyset$ by definition of neighbors. 
Notice that for any pair $(\bm{a}, \bm{b}) \in \mathcal{R}(U)$, if $\check{P}(\bm{a}, \bm{b})=w > 0$
then this either contributes $w$ to both $P_{\cal A}(U)$
and $P_{\mathcal{B}}(\mathcal{N}_{\mathcal{R}}(U))$,
or contributes $0$ to both.
However, if $\check{P}(\bm{a}',\bm{b}') = w>0$ for some $(\bm{a}', \bm{b}') \in \mathcal{R} \cap (\overline{U} \times \mathcal{N}_{\mathcal{R}}(U))$, then $P_{\mathcal{B}}(\mathcal{N}_{\mathcal{R}}(U))$ receives mass $w$ but $P_{\mathcal{A}}(U)$ receives zero mass. Since we have shown that this cannot be offset by mass assigned to any $(\bm{a}, \bm{b}) \in \mathcal{R}(U)$, it follows that under $\check{P}$, $P_{\mathcal{A}}(U) \neq P_{\mathcal{B}}({\cal N}_{\cal R}(U))$, thus this extremal point is not in this face, which is a contradiction. Therefore, we have $\check{P}(\mathcal{R}(U))=1$. If $\check{P}$ is a point mass, then the extremal point is already in $\mathcal{R}(U)$; otherwise, $\check{P}$ is a mixture of point masses, which implies that the extremal point can be written as a convex combination of points in $\mathcal{R}(U)$ and hence, again, a contradiction. 
\end{proof}

\section{Proof of \cref{theo:redundancy}} \label{apx:proof-theo-redun}
In this Appendix, we focus on the bipartite graph associated with the coherence relation ${\cal R}_{C}$ defined in \cref{eqs:coherence}. We use $\Nc(\cdot)$ and $\Nc'(\cdot)$ to denote the set of neighbors for a subset of $\mathcal{A}$ and $\mathcal{B}$ respectively. The relation ${\cal R}_{C}$ has the following property. 

\begin{lemma}\label{lemma:reverse-ineq}
{\rm (1)} For $\calv = \calv^{(1)} \times \dots \times \calv^{(K)} \subseteq \mathcal{A} = [M]^K$, let $B = \overline{\Nc(\calv)} \subseteq \mathcal{B} = [K] \times [M]$. Then, we have $\calv=\overline{\Nc'(B)}$.

{\rm (2)} For $B\subseteq {\cal B}=[K]\times[M]$, let $\calv=\overline{\Nc'(B)}\subseteq {\cal A}=[M]^K$. Then, we have $B=\overline{\Nc(\calv)}$.
\end{lemma}

\begin{proof}

(1) By definition of coherence in ${\cal R}_{C}$, we have 
\[ B = \overline{\Nc(\calv
)} = \bigcup_{i=1}^K \left(\{i\} \times \overline{\calv^{(i)}}\right). \]
It follows that 
\begin{multline*}
\Nc'(B) = \left(\overline{\calv^{(1)}} \times [M] \times \dots \times [M]\right) \cup \left( [M] \times \overline{\calv^{(2)}} \times [M] \dots \times [M]\right) \\
\cup \dots \cup \left( [M] \times  \dots \times [M] \times \overline{\calv^{(K)}} \right).
\end{multline*}
Then, we have by de Morgan's law
\begin{multline*}
\overline{\Nc'(B)} = \left({\calv^{(1)}} \times [M] \times \dots \times [M]\right) \cap \left( [M] \times {\calv^{(2)}} \times [M] \dots \times [M]\right) \\
\cap \dots \cap \left( [M] \times  \dots \times [M] \times {\calv^{(K)}} \right),
\end{multline*}
and hence we have $\calv=\overline{\Nc'(B)}$.

(2) By definition of coherence in ${\cal R}_{C}$, we have 
\[ \calv = \overline{\Nc'(B
)} = \prod_{i=1}^K\calv^{(i)} \text{, where }\calv^{(i)}=\{v|(i,v)\notin B\}. \] 
Then, we have 
\[\overline{\calv^{(i)}}=\{v|(i,v)\in B\}.\]
Thus, it follows that 
\[\overline{\Nc(\calv)}=\bigcup_{i=1}^K \left(\{i\} \times \overline{\calv^{(i)}}\right)=B.\]
\end{proof}

Together, \cref{lemma:reverse-ineq}(1) and \cref{lemma:reverse-ineq}(2) establish that there is a one-to-one correspondence between all $\calv = \calv^{(1)} \times \dots \times \calv^{(K)} \subseteq \mathcal{A}$ and all $B\subseteq {\cal B}$.  Hence, the set of inequalities 
\begin{equation}\label{eq:Binequality}
    P(B) \leq P'(\Nc'(B)) = P'(\overline{\mathcal{V}}), \quad \emptyset \subset B \subset \mathcal{B}.
\end{equation}
is equivalent to the set of inequalities in \cref{theo:1cont}
\[ P'(\mathcal{V}) \leq P(\overline{B}), \quad \emptyset \subset \overline{B} \subset \mathcal{B}. \]

By \cref{prop:redundancy}, the inequality corresponding to $B$ is associated with the set of extremal points described by the set of edges
\begin{equation} \label{eqs:RB}
\Rc(B) = \Big[\Rc \cap (\Nc'(B) \times B) \Big] \; \cup \; \left[\Rc \cap (\overline{\Nc'(B)} \times \overline{B})\right].
\end{equation}
The inequality is redundant iff there exists $B' \neq B$ such that $\Rc(B) \subseteq \Rc(B')$. 

%
\begin{figure}[!htb]
    \centering
        \begin{subfigure}[t]{0.3\textwidth}
\begin{tikzpicture}[scale=0.48, transform shape] 
\node[name=b00, text=blue]{$(Y(x_1)\!=\!1,Y(x_2)\!=\!1)$};
\node[name=b01,below = of b00, yshift=0.6cm, text=blue]{$(Y(x_1)\!=\!1,Y(x_2)\!=\!2)$};
\node[name=b02,below = of b01, yshift=0.6cm, text=blue]{$(Y(x_1)\!=\!1,Y(x_2)\!=\!3)$};
\node[name=b10,below = of b02, yshift=0.6cm, text=blue]{$(Y(x_1)\!=\!2,Y(x_2)\!=\!1)$};
\node[name=b11,below = of b10, yshift=0.6cm, text=red]{$(Y(x_1)\!=\!2,Y(x_2)\!=\!2)$};
\node[name=b12,below = of b11, yshift=0.6cm, text=red]{$(Y(x_1)\!=\!2,Y(x_2)\!=\!3)$};
\node[name=b20,below = of b12, yshift=0.6cm, text=blue]{$(Y(x_1)\!=\!3,Y(x_2)\!=\!1)$};
\node[name=b21,below = of b20, yshift=0.6cm, text=red]{$(Y(x_1)\!=\!3,Y(x_2)\!=\!2)$};
\node[name=b22,below = of b21, yshift=0.6cm, text=red]{$(Y(x_1)\!=\!3,Y(x_2)\!=\!3)$};

\node[name=a00, right= of b00, xshift=1.6cm, text=blue]{{$\boxed{(X\!=\!1,Y\!=\!1)}$}};
\node[name=a01,below = of a00, text=red]{$(X\!=\!1,Y\!=\!2)$};
\node[name=a02,below = of a01, text=red]{$(X\!=\!1,Y\!=\!3)$};
\node[name=a10,below = of a02, text=blue]{ $\boxed{(X\!=\!2,Y\!=\!1)}$};
\node[name=a11,below = of a10, text=red]{$(X\!=\!2,Y\!=\!2)$};
\node[name=a12,below = of a11, text=red]{$(X\!=\!2,Y\!=\!3)$};

\node[name=a,above = 3mm of a00]{${\cal B}$};
\node[name=b,above = 3mm of b00]{${\cal A}$};

\draw[blue,-] (a00.west) -- (b00.east);
\draw[blue,-] (a00.west) -- (b01.east);
\draw[blue,-] (a00.west) -- (b02.east);
\draw[blue,-] (a10.west) -- (b00.east);
\draw[blue,-] (a10.west) -- (b10.east);
\draw[blue,-] (a10.west) -- (b20.east);

\draw[red,-] (a01.west) -- (b11.east);
\draw[red,-] (a01.west) -- (b12.east);

\draw[red,-] (a02.west) -- (b21.east);
\draw[red,-] (a02.west) -- (b22.east);

\draw[red,-] (a11.west) -- (b11.east);
\draw[red,-] (a11.west) -- (b21.east);

\draw[red,-] (a12.west) -- (b12.east);
\draw[red,-] (a12.west) -- (b22.east);

\end{tikzpicture}
        \caption{Example of Case I where $|B|>1$ and $B$ contains more than one $X$-level.}
    \end{subfigure}
    ~ 
    \begin{subfigure}[t]{0.3\textwidth}
\begin{tikzpicture}[scale=0.48, transform shape] 
\node[name=b00,text=blue]{$(Y(x_1)\!=\!1,Y(x_2)\!=\!1)$};
\node[name=b01,below = of b00, yshift=0.6cm, text=blue]{$(Y(x_1)\!=\!1,Y(x_2)\!=\!2)$};
\node[name=b02,below = of b01, yshift=0.6cm, text=blue]{$(Y(x_1)\!=\!1,Y(x_2)\!=\!3)$};
\node[name=b10,below = of b02, yshift=0.6cm, text=red]{$(Y(x_1)\!=\!2,Y(x_2)\!=\!1)$};
\node[name=b11,below = of b10, yshift=0.6cm, text=red]{$(Y(x_1)\!=\!2,Y(x_2)\!=\!2)$};
\node[name=b12,below = of b11, yshift=0.6cm, text=red]{$(Y(x_1)\!=\!2,Y(x_2)\!=\!3)$};
\node[name=b20,below = of b12, yshift=0.6cm, text=red]{$(Y(x_1)\!=\!3,Y(x_2)\!=\!1)$};
\node[name=b21,below = of b20, yshift=0.6cm, text=red]{$(Y(x_1)\!=\!3,Y(x_2)\!=\!2)$};
\node[name=b22,below = of b21, yshift=0.6cm, text=red]{$(Y(x_1)\!=\!3,Y(x_2)\!=\!3)$};

\node[name=a00, xshift=1.6cm, right= of b00, text=blue]{{$\boxed{(X\!=\!1,Y\!=\!1)}$}};
\node[name=a01,below = of a00, text=red]{$(X\!=\!1,Y\!=\!2)$};
\node[name=a02,below = of a01, text=red]{$(X\!=\!1,Y\!=\!3)$};
\node[name=a10,below = of a02, text=red]{$(X\!=\!2,Y\!=\!1)$};
\node[name=a11,below = of a10, text=red]{$(X\!=\!2,Y\!=\!2)$};
\node[name=a12,below = of a11, text=red]{$(X\!=\!2,Y\!=\!3)$};

\node[name=a,above = 3mm of a00]{${\cal B}$};
\node[name=b,above = 3mm of b00]{${\cal A}$};

\draw[blue,-] (a00.west) -- (b00.east);
\draw[blue,-] (a00.west) -- (b01.east);
\draw[blue,-] (a00.west) -- (b02.east);
\draw[red,-] (a01.west) -- (b10.east);
\draw[red,-] (a01.west) -- (b11.east);
\draw[red,-] (a01.west) -- (b12.east);
\draw[red,-] (a02.west) -- (b20.east);
\draw[red,-] (a02.west) -- (b21.east);
\draw[red,-] (a02.west) -- (b22.east);

\draw[red,-] (a10.west) -- (b10.east);
\draw[red,-] (a10.west) -- (b20.east);

\draw[red,-] (a11.west) -- (b11.east);
\draw[red,-] (a11.west) -- (b21.east);

\draw[red,-] (a12.west) -- (b12.east);
\draw[red,-] (a12.west) -- (b22.east);
\end{tikzpicture}
        \caption{Example of Case II where $|B|=1$.}
    \end{subfigure}%
    ~ 
    \begin{subfigure}[t]{0.3\textwidth}
\begin{tikzpicture}[scale=0.48, transform shape] 

\node[name=b00, text=blue]{$(Y(x_1)\!=\!1,Y(x_2)\!=\!1)$};
\node[name=b01,below = of b00, yshift=0.6cm, text=blue]{$(Y(x_1)\!=\!1,Y(x_2)\!=\!2)$};
\node[name=b02,below = of b01, yshift=0.6cm, text=blue]{$(Y(x_1)\!=\!1,Y(x_2)\!=\!3)$};
\node[name=b10,below = of b02, yshift=0.6cm, text=blue]{$(Y(x_1)\!=\!2,Y(x_2)\!=\!1)$};
\node[name=b11,below = of b10, yshift=0.6cm, text=blue]{$(Y(x_1)\!=\!2,Y(x_2)\!=\!2)$};
\node[name=b12,below = of b11, yshift=0.6cm, text=blue]{$(Y(x_1)\!=\!2,Y(x_2)\!=\!3)$};
\node[name=b20,below = of b12, yshift=0.6cm, text=red]{$(Y(x_1)\!=\!3,Y(x_2)\!=\!1)$};
\node[name=b21,below = of b20, yshift=0.6cm, text=red]{$(Y(x_1)\!=\!3,Y(x_2)\!=\!2)$};
\node[name=b22,below = of b21, yshift=0.6cm, text=red]{$(Y(x_1)\!=\!3,Y(x_2)\!=\!3)$};

\node[name=a00, right= of b00, xshift=1.6cm, text=blue]{{ $\boxed{(X\!=\!1,Y\!=\!1)}$}};
\node[name=a01,below = of a00, text=blue]{$\boxed{(X\!=\!1,Y\!=\!2)}$};
\node[name=a02,below = of a01, text=red]{$(X\!=\!1,Y\!=\!3)$};
\node[name=a10,below = of a02, text=red]{ $(X\!=\!2,Y\!=\!1)$};
\node[name=a11,below = of a10, text=red]{ $(X\!=\!2,Y\!=\!2)$};
\node[name=a12,below = of a11, text=red]{ $(X\!=\!2,Y\!=\!3)$};

\node[name=a,above = 3mm of a00]{${\cal B}$};
\node[name=b,above = 3mm of b00]{${\cal A}$};

\draw[blue,-] (a00.west) -- (b00.east);
\draw[blue,-] (a00.west) -- (b01.east);
\draw[blue,-] (a00.west) -- (b02.east);
\draw[blue,-] (a01.west) -- (b10.east);
\draw[blue,-] (a01.west) -- (b11.east);
\draw[blue,-] (a01.west) -- (b12.east);

\draw[red,-] (a02.west) -- (b20.east);
\draw[red,-] (a02.west) -- (b21.east);
\draw[red,-] (a02.west) -- (b22.east);

\draw[red,-] (a10.west) -- (b20.east);

\draw[red,-] (a11.west) -- (b21.east);

\draw[red,-] (a12.west) -- (b22.east);
\end{tikzpicture}
        \caption{Example of Case III where $|B|>1$ but $B$ contains only one $X$-level.}
    \end{subfigure}
\caption{For each $B$ (vertices in box), $\Rc(B)$ consists of both blue edges (between $\Nc'(B)$ and $B$) and red edges (between $\overline{\Nc'(B)}$ and $\overline{B}$). By \cref{lemma:reverse-ineq}, there is a one-to-one correspondence between ${\cal V}\subseteq{\cal A}$ and $B\subseteq{\cal B}$, and we have ${\cal V}=\overline{\Nc'(B)}$. Note that the edges in (c) are contained in those in (a).}
\label{fig:RB}
\end{figure}

We prove \cref{theo:redundancy} in the three parts outlined in \cref{sec:redundancy}, which correspond to sets $B=\overline{\Nc(\calv)}$ such that
$\emptyset \subset B \subset \mathcal{B}$
and $\Nc'(B) \subset \mathcal{A}$,
since $B=\emptyset$, $B={\cal B}$ or 
$\Nc'(B) = \mathcal{A}$ correspond to trivial inequalities.


\begin{enumerate}[(I)]
    \item There exist $k \neq k^{\ast}$ such that $\calv^{(k)}\neq [M]$ and $\calv^{(k^\ast)}\neq [M]$. As shown in \cref{fig:RB}(a), in this case we have that $|B|>1$ and $B$ contains more than one $X$-level. Again we will show there is no $B'\neq B$ such that $\Rc(B) \subseteq \Rc(B')$. \label{item:I}
    \item There exists a single $k^{\ast}$ such that $|\calv^{(k^*)}|=M-1$ and $\calv^{(k)} = [M]$ for every $k \neq k^{\ast}$. As shown in \cref{fig:RB}(b), in this case we have $|B|=1$. We show there is no $B'\neq B$ such that $\Rc(B) \subseteq \Rc(B')$
    for which the inequality \cref{eq:Binequality} is non-trivial. \label{item:II}
\end{enumerate}
In both cases this suffices to establish that the inequality corresponding to $B$ is non-redundant.

\begin{enumerate}[(I)]
\setcounter{enumi}{2}
    \item There exists a single $k^\ast \in[K]$ such that $|\calv^{(k^\ast)}|<M-1$ and $\calv^{(k)} = [M]$ for every $k \neq k^{\ast}$. As shown in 
    \cref{fig:RB}(c), in this case we have $|B|>1$ but $B$ contains only one $X$-level. We will show there exists $B'\neq B$, such that \cref{eq:Binequality} is non-trivial, but  $\Rc(B) \subseteq \Rc(B')$, so that by \cref{prop:redundancy}, 
    the inequality corresponding to $B$ is redundant.\label{item:III}
\end{enumerate}


\subsection{Lemmas}

We first introduce the following lemmas.

\begin{lemma}\label{lemma:A-bar-multix}
If $B$ corresponds to a non-trivial inequality in \cref{eq:Binequality}, then $\overline{B}$ contains every level of $X$.  
\end{lemma}
\begin{proof}
If $B$ leads to a non-trivial inequality, then there is at least one type, \newline{$\left(Y(x_1)=y^1,\dots, Y(x_K)=y^K\right)$}, in $\cal A$ that is not a neighbor of $B$. The set of $y$ values for each $X$-level in $\left(Y(x_1)=y^1,\dots, Y(x_K)=y^K\right)$ satisfies the claim. 
\end{proof}
\begin{lemma}\label{lemma:events-connection}
If there exists a path in $\Rc(B)$ between one point in $\mathcal{B}$ and one point in $\cal A$, then the two points are either in $B$ and $\Nc'(B)$ respectively, or in $\overline{B}$ and $\overline{\Nc'(B)}$ respectively. 
\end{lemma}

\begin{proof} This follows from the definition of $\Rc(B)$.
\end{proof}

   

\begin{lemma}\label{lemma:graph-connection}
    Suppose 
    $\emptyset \subset B \subset \mathcal{B}$. The following hold:
    \begin{enumerate} 
        \item Every pair $(\bm{a},\bm{b})$ with  ${\bm{b}} \in \overline{B}$
and  ${\bm{a}} \in \overline{\Nc'{(B)}} $ 
is connected by a path in $\Rc(B)$.
        \item  If $|B|>1$ and $B$ contains points with more than one $X$-level, then between every point $\bm{b} \in B$ and every point $\bm{a} \in \Nc'{(B)}$ there exists a path in $\Rc(B)$.
    \end{enumerate}
\end{lemma}
\begin{proof}

   Observe that by definition, in $\Rc(B)$ all points $(X=\alpha^*, Y=\beta^*)$ in $\overline{B}$ are adjacent to all points in $\overline{\Nc'(B)}\cap \{Y(\alpha^*)=\beta^*\}$.
Similarly, in $\Rc(B)$, all points $(X=\alpha', Y=\beta')$ in $B$ are adjacent to $\Nc'(B)\cap \{Y(\alpha')=\beta'\}$. 
    

\begin{enumerate}
\item If ${\Nc'(B)}={\cal A}$, then the claim follows trivially. Otherwise, by
\cref{lemma:A-bar-multix}, 
$\overline{B}$ contains more than one $X$-level.
 
Let $(X=\alpha_1, Y=\beta_1)\in\overline{B}$. We know in $\Rc(B)$, we have $(X=\alpha_1, Y=\beta_1)\leftrightarrow\overline{\Nc'(B)}\cap \{Y(\alpha_1)=\beta_1\}$. 
Therefore, it is sufficient to prove there is a path connecting $(X=\alpha_1, Y=\beta_1)$ to each point in  $\overline{\Nc'(B)}\cap \{Y(\alpha_1)\neq\beta_1\}$. 
For an arbitrary type $\bm{a} = (Y(\alpha_1)=\gamma, Y(\alpha_2)=\beta_2, \dots)\in\overline{\Nc'(B)}\cap \{Y(\alpha_1)\neq \beta_1\}$ where $\gamma\neq \beta_1$, we know $(X=\alpha_1, Y=\gamma), (X=\alpha_2, Y=\beta_2)\in\overline{B}$, since otherwise the type $\bm{a}$ would be in $\Nc'(B)$.
Hence, we have an edge $(X=\alpha_2, Y=\beta_2)\leftrightarrow
\bm{a}$ 
in $\Rc(B)$ since it is connecting $\overline{B}\leftrightarrow\overline{\Nc'(B)}$. 
Let $\bm{a}^*$ be the type corresponding to $\bm{a}$ but replacing 
$\gamma$ with $\beta_1$.
Since $(X=\alpha_1, Y=\beta_1) \in\overline{B}$ and $\bm{a} \in\overline{\Nc'(B)}$, we have $\bm{a}^*=(Y(\alpha_1)=\beta_1, Y(\alpha_2)=\beta_2, \dots)\in\overline{\Nc'(B)} $, so we have an edge $(X=\alpha_2, Y=\beta_2)\leftrightarrow \bm{a}^*$ in $\Rc(B)$. 
Hence, in $\Rc(B)$, we have $\bm{a}\leftrightarrow(X=\alpha_2,Y=\beta_2)\leftrightarrow \bm{a}^* \leftrightarrow(X=\alpha_1, Y=\beta_1)$. 
Since $\bm{a}$ is arbitrary,
the conclusion follows.
\item 
Since any point in $\Nc'(B)$ is connected to at least one point in $B$, it suffices to show there exists a path in $\Rc(B)$ between every pair of events
$\bm{b}_1,\bm{b}_2\in B$.
Consider $(X=\alpha_1, Y=\beta_1)$ and $(X=\alpha_2, Y=\beta_2)$ in $B$. If $\alpha_1\neq \alpha_2$, then we have $(X=\alpha_1, Y=\beta_1)\leftrightarrow(Y(\alpha_1)=\beta_1, Y(\alpha_2)=\beta_2,\dots )\leftrightarrow (X=\alpha_2, Y=\beta_2)$, since $(Y(\alpha_1)=\beta_1, Y(\alpha_2)=\beta_2,\dots)\in\Nc'((X=\alpha_1, Y=\beta_1))\cap \Nc'((X=\alpha_2, Y=\beta_2))$. If $\alpha_1= \alpha_2$, then there exists a point $(X=\alpha_3, Y=\beta_3)\in B$ where $\alpha_3\neq \alpha_1=\alpha_2$ since by hypothesis $B$ contains points with more than one $X$-level. Since $\alpha_1\neq \alpha_3 \neq \alpha_2$, we know $(X=\alpha_3, Y=\beta_3)$ is connected with both $(X=\alpha_1, Y=\beta_1)$ and $(X=\alpha_2, Y=\beta_2)$ by similar arguments as above. Hence, $(X=\alpha_1, Y=\beta_1)$ and $(X=\alpha_2, Y=\beta_2)$ are connected as well. 
\end{enumerate}

    
\end{proof}

\begin{remark}\label{remark:graph-connection}
\cref{lemma:graph-connection}.1 directly implies that any two points in $\overline{B}$ (or $\overline{\Nc'(B)}$) are connected by a path in $\Rc(B)$. Similarly, \cref{lemma:graph-connection}.2 implies that if $|B|>1$ and $B$ contains points with more than one $X$-level, then any two points in $B$ (or $\Nc'(B)$) are connected by a path in $\Rc(B)$.  
\end{remark}



\subsection{Proof of (I)}

\begin{proof}
We consider two cases: i. $B'\not\subset B$, and ii. $B' \subset B$.
\begin{description}
\item[Case i.] Let $B'\not\subset B$. Suppose for a contradiction that $\Rc(B) \subseteq \Rc(B')$. We will show that the inequality induced by $B'$ is trivial which is a contradiction. Since $B'\not\subset B$, there exists $\bm{b}'\in B'$ such that $\bm{b}'\in \overline{B}$. Let $\bm{b}':=(X=i, Y=y^{i})$, and $A:=\Nc'{(\bm{b}')}=\{\left(Y(x_1)=\tilde{y}^{1},\dots, Y(x_i)=y^{i},\dots, Y(x_K)=\tilde{y}^{K}\right): \tilde{y}^{1},\dots, \tilde{y}^{i-1},\tilde{y}^{i+1},\dots, \tilde{y}^{K}\in[M]\}$. Note that $A\subseteq \Nc'(B')$.

We partition $A$ into $A_1$ and $A_2$ such that $A_1:= \Nc'(B)\cap A$, $A_2:= A\setminus\Nc'(B)\subseteq\Nc'(B')\cap\overline{\Nc'(B)}$. We further claim that $A_1, A_2\neq \emptyset$. We first show $A_1$ is non-empty. Since by hypothesis $B$ contains points with at least two $X$-levels, there exists a point $ (X=k, Y=y^k)$ in $B$ such that $k\neq i$. Hence, we have $(Y(x_k)=y^k, Y(x_i)=y^i,\dots)\in\Nc'(B)\cap A$ which is in $A_1$. Now we show $A_2\neq \emptyset$ by showing there exists $\bm{a}$ such that $\bm{a}\in \Nc'(\bm{b}')$ and $\bm{a}\not\in\Nc'(B)$. By \cref{lemma:A-bar-multix}, we know for all $x\in[K]$, there exists a point $(X=x, Y=y)\not\in B$.
We further know $\bm{b}'=(X=i, Y=y^i)\not\in B$. Thus we have points $(X=1, Y=y^1), \dots, (X=i, Y=y^i),\dots, (X=K, Y=y^K)$ in $\overline{B}$. Then, we have $\bm{a}=(Y(x_1)=y^1, \dots, Y(x_i)=y^i,\dots, Y(x_K)=y^K)\in \Nc'(\bm{b}')$ but not in $\Nc'(B)$ as desired. 

Now, we will show $B\subseteq B'$ and $\overline{B}\subseteq B'$ to establish the contradiction that $B'=\cal B$. 
By \cref{lemma:graph-connection}.2
there is a path connecting any $\bm{a}_1\in A_1\subseteq \Nc'(B)$ to all $\bm{b}\in B$ in $\Rc(B)$, thus also in $\Rc(B')$. Since $A_1\subseteq {\Nc'(\bm{b}')}$, we have,
by \cref{lemma:events-connection}, that
all $\bm{b}\in B$ are in $B'$, i.e., $B\subseteq B'$.
By construction, 
since $A_2 \subseteq \Nc'(\bm{b}')$, in $\Rc(B')$
there are edges connecting $\bm{b}'\leftrightarrow \bm{a}$ for all $\bm{a}\in A_2$;  these are edges connecting $\overline{B}\leftrightarrow\overline{\Nc'(B)}$ in $\Rc(B')$. Then, by \cref{lemma:graph-connection}.1, we know for any point in $A_2\subseteq \overline{\Nc'(B)}$, there exists a path in $\Rc(B)$, and thus also in $\Rc(B')$, that connects to $\overline{\bm{b}},$ for all $\overline{\bm{b}}\in \overline{B}$. Note that since $A_2\subseteq \Nc'(B')$, we have,
by \cref{lemma:events-connection}, that 
$\overline{\bm{b}}\in \overline{B}$ implies 
$\overline{\bm{b}}\in B'$ so $\overline{B}\subseteq B'$. Thus we have $\overline{B}\cup B\subseteq{B'}$, so $B'=\cal{B}$ which leads to a trivial inequality.

\item[Case ii.] Let $B'\subset B$. 
To show that $\Rc(B) \not\subseteq \Rc(B')$, it is sufficient to show $\Rc(B)$ contains edges $\overline{B'}\leftrightarrow \Nc'(B')$ which are by construction not in $\Rc(B')$. Since $B'\subset B$, there exists $\bm b$ such that 
$\bm{b}\in B \cap \overline{B'}$. 

If $B'$ does not contain points with all levels of $X$ present in $B$, then there exists $\bm{b}\in B \cap \overline{B'}$ with point $(X=x', Y=y')$ where there is a point with $X=x'$ in $B$ but no point with $X=x'$ is present in $B'$. Let $\left(X=x_1, Y=y^1\right)$ be a point in $B'$. Therefore, there exists a point $\left(Y(x_1)=y^{1}, \dots, Y(x')=y',\dots\right)$ in $\Nc'(\bm{b})\subset \Nc'(B)$, which is also in $\Nc'(B')$. Hence, we have an edge in $\Rc(B)$, $(X=x', Y=y')\leftrightarrow \left(Y(x_1)=y^{1},\dots, Y(x')=y',\dots\right)$, that connects $\overline{B'}$ and $\Nc'(B')$ which is not in $\Rc(B')$. 

Now suppose $B'$ contains points with all levels of $X$ present in $B$. Let $\bm{b}_2
= \left(X\!=\!x_2, Y\!=\!y^{2}\right)$ with $\bm{b}_2
\in B\cap \overline{B'}$. Since $B$ contains at least two $X$-levels and $B'$ contains all levels of $X$ in $B$, let $(X=x^\dagger, Y=y^\dagger)\in B'$ such that $x^\dagger\neq x_2$. Therefore, we have $\left(Y(x_2)=y^{2}, Y(x^\dagger)=y^\dagger,\dots\right)\in \Nc'(\bm{b}_2)\subset \Nc'(B)$. Since $(X=x^\dagger, Y=y^\dagger)\in B'$, we also have $\left(Y(x_2)=y^{2}, Y(x^\dagger)=y^\dagger,\dots\right)\in \Nc'(B')$. Hence, we have an edge in $\Rc(B)$, $\bm{b}_2= \left(X=x_2, Y=y^{2}\right)\leftrightarrow \left(Y(x_2)=y^{2}, Y(x^\dagger)=y^\dagger,\dots\right)$ that connects $\overline{B'}$ and $\Nc'(B')$ but which is thus not in $\Rc(B')$.
\end{description}

\end{proof}

\subsection{Proof of (II)}

\begin{proof}
Since $|B|=1$, suppose
without much loss of generality that $B=\{(X=1, Y=1)\}$. It follows that $$\Nc'(B)=\left\{\left(Y(x_1)=1, Y(x_2)=y^{2},\dots, Y(x_K)=y^{K}\right):y^{2},\dots, y^{K}\in [M] \right\}.$$ 

Suppose for a contradiction that there exists a set $B'$, $\emptyset \subset B' \subset \mathcal{B}$ such that $\Rc(B') \supseteq \Rc(B)$. 
Since $B' \neq B$, $B'$ contains at least one other point not in $B$.
We first show that $B'$ contains a point $(X=i,Y=y')$ where $i\neq 1$. Suppose, for a contradiction, that $B'$ only contains points with $X=1$.
Again without much loss of generality, suppose $B'$ contains the point $(X=1, Y=2)$. Let $x^*\in [K]\setminus \{1\}$ be any other level of $X$. Then in $\Rc(B)$ there is an edge $(X=x^*, Y=y^*)\leftrightarrow(Y(x_1)=2, \ldots, Y(x^*)=y^*,\dots)$ connecting $\overline{B}$ to $\overline{\Nc'(B)}$. However, since 
$B'$ only contains events with $X=1$, 
$(X=x^*, Y=y^*)\notin B'$. Since $(Y(x_1)=2, \ldots, Y(x^*)=y^*,\dots)\in\Nc'(B')$, this edge is not in $\Rc(B')$, which  contradicts $\Rc(B') \supseteq \Rc(B)$. Hence, we know $B'$ contains a point $(X=i,Y=y')$ where $i\neq 1$ and again without much loss of generality, we suppose $i=2$ so $(X=2, Y=y')\in B'$. 



Let $B_1=\{(X=1, Y=1),\dots,(X=1, Y=M)\}=\{\bm b_1, \dots, \bm b_M\}$. Note that $\Nc'(B_1)={\cal A}$. We first show for all points in $B_1$, all edges connecting $\bm{b}_i\leftrightarrow \Nc'(\bm{b}_i)$ for all $i\in[M]$ are in $\Rc(B)$ and thus, by hypothesis, are also in $\Rc(B')$. By definition, $B=\{(X=1, Y=1)\}=\{\bm b_1\}$ is connected to all of its neighbors in $\Rc(B)$ since these are edges connecting $B\leftrightarrow \Nc'(B)$. In addition, we know that $\bm{b}_2,\dots, \bm{b}_M\in B_1$ are connected to all of their neighbors in $\Rc(B)$ since, by definition of coherence \eqref{eqs:coherence}, 
$\Nc'(\bm{b}_i)\cap \Nc'(\bm{b}_j)=\emptyset$ for all $i\neq j$ and thus, since $B=\{\bm{b}_1\}$ these are edges connecting $\overline{B}\leftrightarrow \overline{\Nc'(B)}$.

We next show $\{(X=2, Y=y')\}\cup B_1\subseteq B'$. Firstly, we know that in $\Rc(B')$, there exists a path $(X=2, Y=y')\leftrightarrow (Y(x_1)=1, Y(x_2)=y', \dots)\leftrightarrow (X=1, Y=1)$, where the first edge is connecting $B'\leftrightarrow\Nc'(B')$ and the second edge is by $\Rc(B')\supseteq\Rc(B)$. Also, we know $\bm{b}_2,\dots,\bm b_M\in \overline{B}$ are connected to $(X=2, Y=y')\in \overline{B}$ by \cref{lemma:graph-connection}.1 and \cref{remark:graph-connection}.
Hence, we know there exists a path in $\Rc(B')$ connecting $(X=2, Y=y')$ to all points in $B_1$. Since $(X=2, Y=y')\in B'$, it follows 
by \cref{lemma:events-connection}
that $\{(X=2, Y=y')\}\cup B_1\subseteq B'$. Hence, we have $B_1\subseteq B'$. By the contrapositive of \cref{lemma:A-bar-multix}, $B'$ corresponds to a trivial inequality with $\Nc'(B')={\cal A}$, which is a contradiction. 


\end{proof}
\vspace{-1cm}

\subsection{Proof of (III)}

\begin{proof}
Let $|B|=r>1$, and suppose $B$ only contains points that have the same level of $X=x$. 
Without much loss of generality, assume $B=\{(X=x_1, Y=y^{1}), (X=x_1, Y=y^{2}),\ldots (X=x_1, Y=y^r)\}=\{\bm{b}_1, \bm{b}_2, \ldots \bm{b}_r\}$. We know 
by definition of coherence \eqref{eqs:coherence},
$\Nc'(\bm{b}_i)\cap \Nc'(\bm{b}_j)=\emptyset$ for $i\neq j$. Let $B'=\{\bm{b}_1\}$. It is sufficient to show that $\Rc(B) \subseteq \Rc(B')$. Recall that $\Rc(B)$ contains edges connecting $B$ and $\Nc'(B)$ as well as edges connecting $\overline{B}$ and $\overline{\Nc'(B)}$. 

First, consider edges connecting $B$ and $\Nc'(B)$. We have $\bm{b}_1\leftrightarrow \Nc'(\bm{b}_1)$ in $\Rc(B')$ since they are edges connecting $B'$ and $\Nc'(B')$. We also have edges $\bm{b}_2\leftrightarrow \Nc'(\bm{b}_2)$ in $\Rc(B')$ since $\bm{b}_2 \in \overline{B'}$, $\Nc'(\bm{b}_2)\subset  \overline{\Nc'(B')}$ since $\Nc'(\bm{b}_1)\cap \Nc'(\bm{b}_2)=\emptyset$, and the same argument can be repeated for edges $\bm{b}_3\leftrightarrow \Nc'(\bm{b}_3)$, etc. Therefore, all edges connecting $B\leftrightarrow \Nc'(B)$ are in $\Rc(B')$. 

Now consider edges connecting $\overline{B}$ and $\overline{\Nc'(B)}$. Since $B'\subset B$ and $\Nc'(\bm{b}_i)\cap \Nc'(\bm{b}_j)=\emptyset$ for $i\neq j$, we have ${\Nc'(B')}\subset{\Nc'(B)}$, and thus $\overline{\Nc'(B)}\subset\overline{\Nc'(B')}$. Note that we also have $\overline{B}\subseteq\overline{B'}$. Hence, the edges connecting $\overline{B}\leftrightarrow\overline{\Nc'(B)}$ are also edges connecting $\overline{B'}\leftrightarrow\overline{\Nc'(B')}$ and are thus in $\Rc(B')$. Therefore, all edges in $\Rc(B)$ are in $\Rc(B')$.
\end{proof}

\vspace{-0.5cm}
\section{Convex program for statistical inference}\label{supp:algo-1}\vspace{-0.2cm}

We present the convex optimization program for statistical inference.

\begin{algorithm}[H]
\caption{Convex program for statistical inference} \label{alg:convex}
\begin{algorithmic}[1]
  \REQUIRE Linear functionals $\{f_j: j \in [J]\}$ of the counterfactual distribution; Matrices $H'$ and $H$; Confidence level $\alpha$;  Empirical probabilities $\hat{p}_{z}\equiv \hat{P}(X,Y\mid Z=z)$ for $z \in [Q]$; Sample size $n_z$ of instrument arm $z \in [Q]$. 
  \STATE \textbf{Variables:} 
  \[\begin{aligned}
  &p_z := P(X,Y\mid Z=z)\in \mathbb{R}^{KM}, \quad z \in [Q]\\
  &p' := P'(Y(x_1),\dots, Y(x_K))\in \mathbb{R}^{M^K}
  \end{aligned}\]
  
  \STATE Determine $t_{\alpha}$ from \cref{eqs:tail-bound} with line search.
  
  \STATE For each $j \in [J]$, solve the following convex program:
  \[
    \begin{aligned}
      \quad & l_j=\min f_j(p'), \quad u_j=\max f_j(p') \\
      \text{s.t.}\quad & -H p_{z}+H' p' \le 0,\quad z \in [Q]\\
                      & \sum_{z=1}^Q n_z \KL(\hat{p}_z \| p_z)\leq t_\alpha,\\
                      & p_z \in \Delta^{KM-1},\quad z \in [Q]\\
                      & p' \in \Delta^{M^K-1}.
    \end{aligned}
  \]
  \RETURN Confidence intervals $[l_j, u_j]$ for $j \in [J]$ (if the feasible region is empty, let $l_j=+\infty$ and $u_j=-\infty$).  
\end{algorithmic}
\end{algorithm}

\section{V- and H-representation of the IV model} \label{sec:cdd}

\cref{theo:redundancy} describes the polytope that characterizes the IV model in the H-representation, namely as the intersection of a finite number of half-spaces. The same polytope can also be described in the V-representation as the convex hull of a finite number of vertices (i.e., extreme points). The vertices, as given by \cref{prop:redundancy}, correspond to the edges of a relation $\Rc \subset \mathcal{A} \times \mathcal{B}$, where $\mathcal{A} = [M]^K$ and $\mathcal{B} = [K] \times [M]$. In what follows, we demonstrate how to obtain the V-representation and convert it to an H-representation. This gives the matrices $H, H'$ used in \cref{alg:convex}.

\begin{example}[Binary IV]
Consider the setting with a binary treatment $X$ and outcome $Y$. Let $Z$ be fixed to a level $z$. Consider a vector in $\mathbb{R}^{8}$ with coordinates defined as follows:
\begin{itemize}
\item the first four coordinates describe the principal strata probabilities $P(Y(x_1)=1, Y(x_2)=1)$, $P(Y(x_1)=1, Y(x_2)=2)$, $P(Y(x_1)=2, Y(x_2)=1)$, $P(Y(x_1)=2, Y(x_2)=2)$, 
\item the next four coordinates describe the observed distribution in instrument arm $z$: $P(X=1, Y=1\mid Z=z)$, $P(X=1, Y=2\mid Z=z)$, $P(X=2, Y=1\mid Z=z)$, $P(X=2, Y=2\mid Z=z)$.
\end{itemize}

The V-representation can be obtained by considering degenerate distributions that assign probability one to a single principal stratum probability and to an observed value coherent with that stratum, and then assign 0 to all other coordinates. Since $X$ is binary, there are two observed outcomes coherent with each principal stratum, and the resulting V-representation is an $8 \times 8$ binary matrix:
\[
\renewcommand{\arraystretch}{1.2} 
\setlength{\arraycolsep}{8pt}      
\begin{bmatrix}
1 & 0 & 0 & 0 & 1 & 0 & 0 & 0 \\[-2pt]
0 & 1 & 0 & 0 & 0 & 1 & 0 & 0 \\[-2pt]
0 & 0 & 1 & 0 & 1 & 0 & 0 & 0 \\[-2pt]
0 & 0 & 0 & 1 & 0 & 1 & 0 & 0 \\[-2pt]
1 & 0 & 0 & 0 & 0 & 0 & 1 & 0 \\[-2pt]
0 & 1 & 0 & 0 & 0 & 0 & 1 & 0 \\[-2pt]
0 & 0 & 1 & 0 & 0 & 0 & 0 & 1 \\[-2pt]
0 & 0 & 0 & 1 & 0 & 0 & 0 & 1
\end{bmatrix},\]
where the first row above encodes the edge between the principal stratum $(Y(x_1)=1, Y(x_2)=1)$ (``always recover'') and the coherent observation $(X=1,Y=1)$. With polyhedral computation tools such as \texttt{cddlib} \citep{cddlib}, this V-representation can be converted to the following H-representation:
\[
\renewcommand{\arraystretch}{1.2} 
\setlength{\arraycolsep}{8pt}      
[-H', H] = \begin{bmatrix*}[r]
  1  &  1  &  0  &  0  &  0  &  0  & -1  &  0 \\[-2pt]
  0  &  0  & -1  &  0  &  1  &  0  &  0  &  1 \\[-2pt]
  0  &  1  &  0  &  1  &  0  & -1  &  0  &  0 \\[-2pt]
  1  &  0  &  1  &  0  & -1  &  0  &  0  &  0 \\[-2pt]
  0  & -1  &  0  &  0  &  0  &  1  &  1  &  0 \\[-2pt]
 -1  &  0  &  0  &  0  &  1  &  0  &  1  &  0 \\[-2pt]
  1  &  1  &  1  &  0  & -1  &  0  & -1  &  0 \\[-2pt]
 -1  & -1  &  0  &  0  &  1  &  1  &  1  &  0 
\end{bmatrix*}, \qquad \]
whose first four columns make $-H'$ and the last four columns make $H$. These matrices, which do not depend on $z$, encode the inequalities 
\[ -H p_{z}+H' p' \le 0,\quad z \in [Q] \]
in \cref{alg:convex}.
\end{example}

\medskip Following the example, we use \texttt{cddlib} to directly compute the number of non-redundant inequalities. \cref{tab:cdd-check} lists the number of inequalities per instrument arm
\[ r / Q = (2^M-1)^K-K(2^M-M-2)-1 \]
under various settings of $(K,M)$. The bold entries, which can be computed with \texttt{cddlib} relatively quickly, have been verified.\footnote{\texttt{cddlib} returns $r/Q+KM+M^K$ inequalities and two equalities. Specifically, there are $KM$ inequalities for $P(X=x, Y=y)\geq 0$, $x\in[K], y\in[M]$, as well as $M^K$ inequalities for $P(Y(x_1)=y^1, \ldots, Y(x_K)=y^K)\geq 0$, $y^1,\dots,y^K\in[M]$. The two equalities encode $\sum_{x=1}^K\sum_{y=1}^MP(X=x, Y=y)=1$ and $\sum_{y^1=1}^M\cdots\sum_{y^K=1}^MP(Y(x_1)=y^1, \ldots, Y(x_K)=y^K)=1$.}

\begin{table}[h]
\caption{Number of non-redundant inequalities per instrument arm under different $M$ (levels of $Y$) and $K$ (levels of $X$). The bold numbers have been verified by \texttt{cddlib}.}
\label{tab:cdd-check}
\centering
\begin{tabular}{crrrrr}
\hline
 & $M=2$ & $M=3$ & $M=4$ & $M=5$ & $M=6$ \\ \hline
 $K=2$ & {\bf 8} & {\bf 42} & {\bf 204} & {\bf 910} & {\bf 3856} \\ \hline
 $K=3$ & {\bf 26} & {\bf 333} & {\bf 3344} & 29715 & 249878 \\ \hline
 $K=4$ & {\bf 80} & {\bf 2388} & 50584 & 923420 & 15752736 \\ \hline
 $K=5$ & {\bf 242} & 16791 & 759324 & 28629025 & 992436262 \\ \hline
\end{tabular}
\end{table}

\section{Run time for the Minneapolis Domestic Violence Experiment results}\label{sec:runtime}

\Cref{tab:analysis_runtime} reports the time (in seconds on an ARM64 personal computer) taken to compute the results presented in \cref{tab:analysis}. The run time for the plug-in bounds depends on $M$ (levels of $Y$), $K$ (levels of $X$), and $Q$ (levels of $Z$), whereas the run time for the confidence intervals further depends on the sample size $n$ and the confidence level $\alpha$.

\begin{sidewaystable}[p]
\caption{Run time (in seconds) for the Minneapolis Domestic Violence Experiment results}\label{tab:analysis_runtime}
\centering
\small
\begin{tabular}{cccccccc} 
\toprule
& \multirow{2}{*}{\textbf{}}  & \multicolumn{2}{c}{\textbf{Advise vs. Arrest}}     & \multicolumn{2}{c}{\textbf{Separate vs. Arrest}}   & \multicolumn{2}{c}{\textbf{Separate vs. Advise}} \\ \hline  Researcher &   & \textbf{Plug-in} & \textbf{CI (95\%)} & \textbf{Plug-in} & \textbf{CI (95\%)} & \textbf{Plug-in} & \textbf{CI (95\%)} \\ \hline
1  & \textbf{All data}   & 0.416	& 5.594 &	0.431 &	5.319	& 0.357	& 5.367  \\ \hline
2   & \textbf{\begin{tabular}[c]{@{}c@{}}Delete $X\!=\!\text{Sep}$,\\ $X\!=\!\text{Adv}$, \\  or $X\!=\!\text{Arr}$ \end{tabular}}  & 0.395	& 4.951	& 0.303	& 4.934	& 0.329	& 4.700  \\ \hline
3    & \textbf{\begin{tabular}[c]{@{}c@{}}Binary IV model: \\ Delete $X\!=\!\text{Sep}$ and $Z\!=\!\text{Sep}$,\\ $X\!=\!\text{Adv}$ and $Z\!=\!\text{Adv}$,\\  or $X\!=\!\text{Arr}$ and $Z\!=\!\text{Arr}$\end{tabular}} & 0.351	& 5.340	& 0.262	& 5.301	&  0.268 &	5.309       \\ \hline
\multirow{3}{*}{4} & \textbf{Delete $Z\!=\!\text{Arr}$}  & 0.271 &	4.909	& 0.362	& 4.775	& 0.360	& 4.780 \\ & \textbf{Delete $Z\!=\!\text{Adv}$}   & 0.278 &	4.732	& 0.270	& 4.738	& 0.268 &	5.015 \\ & \textbf{Delete $Z\!=\!\text{Sep}$}  & 0.352 &	4.904	& 0.275	& 4.773 &	0.282	& 4.747  \\ \bottomrule
\end{tabular}
\end{sidewaystable}

\section{Falsification of the IV model} \label{sec:feasibility}
Given the observed distribution $P(X,Y \mid Z=z)$ across instrument arms $z \in [Q]$, one can conduct a falsification test of the IV model by checking feasibility of the inequalities in \cref{theo:redundancy} (ignoring sampling variability). Interior-point methods take time that is polynomial in $Q$ to check feasibility \citep{nemirovski2008interior}. For illustration, we simulate $P(X,Y \mid Z=z)$ for each $z$ from $\text{Dirichlet}(1,\dots,1)$ under $M=2$. \cref{fig:sim} shows the proportion of instances that would falsify the IV models as $K$ and $Q$ vary. Because the inequalities take the form of an intersection across instrument arms, as one would expect, the proportion grows with the number of instrument arms.

\begin{figure}[!htb]
    \centering
    \includegraphics[width=0.7\textwidth]{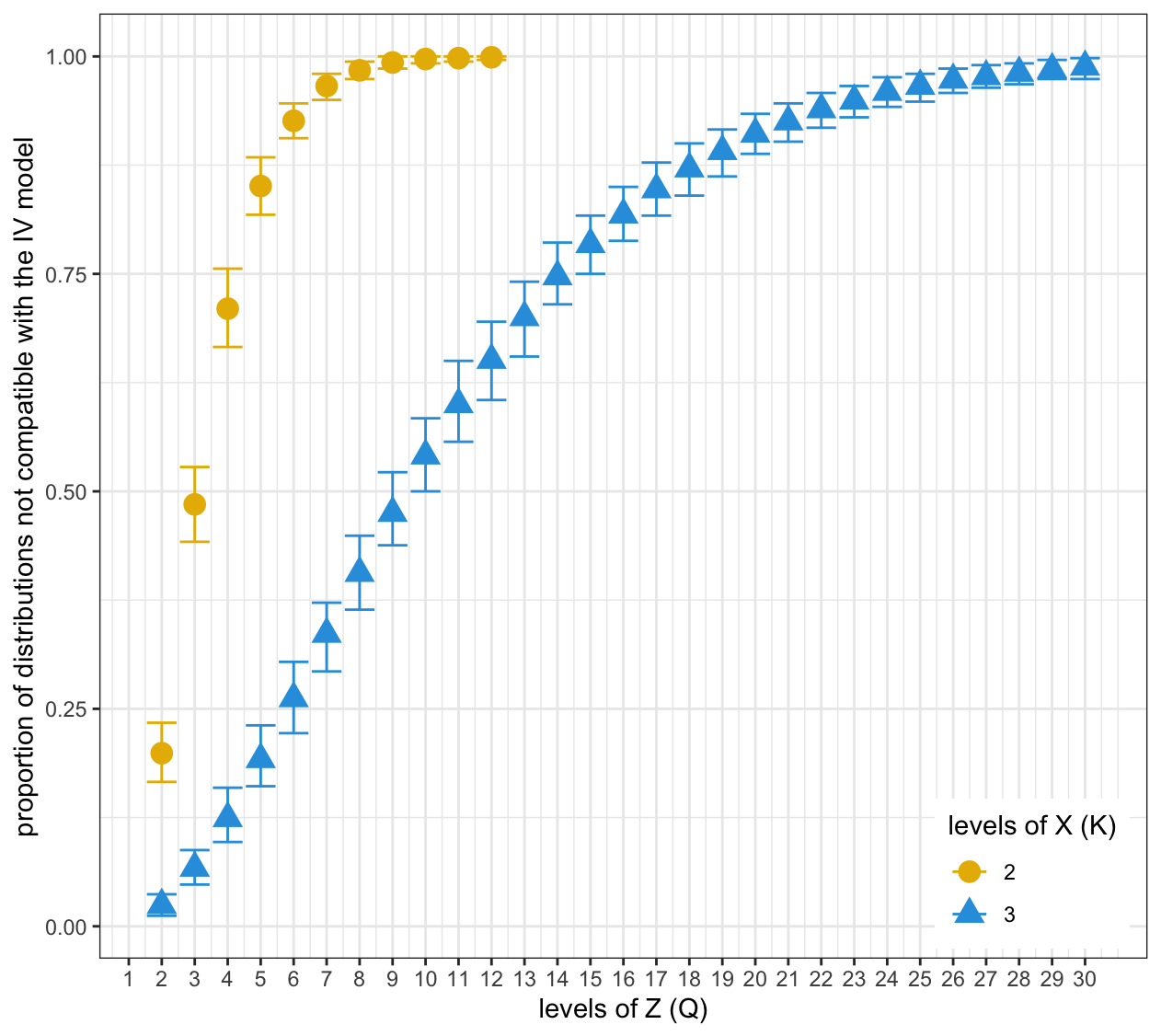}
    \caption{The proportion of instances that would falsify the IV models when the observed distribution $P(X,Y \mid Z=z)$ is generated from $\text{Dirichlet}(1,\dots,1)$ for each $z$. }
    \label{fig:sim}
\end{figure}

In fact, instead of checking that the intersection of all arms is non-empty, one can check that the intersection of every combination of $M^K$ arms is non-empty, as ensured by the next theorem. This can be convenient when $M^K \ll Q$.

\begin{theorem}[Helly's theorem]
Let $C_1, \dots, C_m$ be a collection of convex subsets of $\mathbb{R}^d$ with $m \geq d+1$. Then, it holds that 
\[ \bigcap_{i=1}^m C_i\neq \emptyset \quad \iff \quad \bigcap_{i \in I} C_i \neq \emptyset, \; \forall I \subset [m], |I| = d+1. \]
\end{theorem}

\section{Discussion of inequalities in \citet{Russell}} \label{appendix:russell}
\citet{Russell} presented ``sharp bounds'' on any continuous functional of the joint counterfactual distribution under the IV model $\calM_1$.
Extending the work of \citet{beresteanu}, \citet{Russell} used results in \citet{LuoWang}
in an attempt to further eliminate the redundant inequalities implied by Artstein's theorem and obtain constraints in the so-called ``exact core determining class'' defining the joint counterfactual distribution. The inequalities in the ``exact core determining class'' bound the counterfactual probabilities 
\begin{equation} \label{ineq:russel}
\begin{cases} 
    P(Y(x_{s_1})=y_1, Y(x_{s_2})=y_2, \ldots, Y(x_{s_K})\in \mathcal{Y}),\ \forall \mathcal{Y}\subseteq [M], & K>2 \text{ or } M \leq K \\
    P(Y(x_i)=y_i, Y(x_j)\in \mathcal{Y}'),\ \forall \mathcal{Y}'\subseteq [M], &K=2 \text{ and } M>K
  \end{cases},
\end{equation}
where $(s_1,\dots, s_K)$ is any permutation of $(1,\dots, K)$. Note that when $K=M=2$, the inequalities in (\ref{ineq:russel}) are the same as \eqref{theo:1cont} and the result in \citet{RichardsonRobins}. However, they deviate from our results when $K\neq 2$ or $M\neq 2$: \eqref{ineq:russel} is a strict subset of the non-redundant constraints given by our \cref{theo:redundancy}. 

As an illustrative example, consider the case with $K=2, M=3$, and $Q=2$. In this case, the set of non-redundant inequalities consists of the following four groups:
\begin{multline} \label{eqs:set-1}
P(Y(x_i)=y_1^{x=i})+P(Y(x_i)=y_2^{x=i})\leq 1-P(X=i, Y=y_3^{x=i}\mid Z=z), \\
y_1^{x=i}\neq y_2^{x=i}\neq y_3^{x=i},
\end{multline}
\begin{multline} \label{eqs:set-2}
P(Y(x_1)=y_1^{x=1}, Y(x_2)=y_1^{x=2})+P(Y(x_1)=y_1^{x=1}, Y(x_2)=y_2^{x=2})\\
+P(Y(x_1)=y_2^{x=1}, Y(x_2)=y_1^{x=2})+P(Y(x_1)=y_2^{x=1}, Y(x_2)=y_2^{x=2})\\
\leq 1-P(X=1, Y=y_3^{x=1}\mid Z=z)-P(X=2, Y=y_3^{x=2}\mid Z=z), \\
y_1^{x=1}\neq y_2^{x=1}\neq y_3^{x=1}, y_1^{x=2}\neq y_2^{x=2}\neq y_3^{x=2},
\end{multline}
\begin{multline} \label{eqs:set-3}
P(Y(x_i)=y^{i}, Y(x_j)=y^{j})+P(Y(x_i)=y^{i}, Y(x_j)=\tilde{y}^{x=j})\\
\leq P(X=i, Y=y^{i}\mid Z=z)+P(X=j, Y=y^{j}\mid Z=z)+P(X=j, Y=\tilde{y}^{x=j}\mid Z=z), \\
y^{j}\neq \tilde{y}^{x=j},
\end{multline}
and
\begin{multline} \label{eqs:set-4}
P(Y(x_1)=y^{1}, Y(x_2)=y^{2})\leq P(X=1, Y=y^{1}\mid Z=z)+P(X=2, Y=y^{2}\mid Z=z).
\end{multline}
Here the inequalities (\ref{eqs:set-2}), (\ref{eqs:set-3}), (\ref{eqs:set-4}) correspond to \cref{theo:redundancy}'s Condition 1, while (\ref{eqs:set-1}) corresponds to Condition 2.
Owing to symmetry, for each level of $z$ there are, respectively,
$3\cdot 2=6$, $3\cdot 3 = 9$, $3\cdot 3\cdot 2=18$, $3\cdot 3 =9$ inequalities in each group (\ref{eqs:set-1})--(\ref{eqs:set-4}). Since here $Z$ is binary, 
our \cref{theo:redundancy} gives $42 \times 2=84$ non-redundant inequalities in total.  However, since only \eqref{eqs:set-3} and \eqref{eqs:set-4} are in the ``exact core determining class'' given by \citet{Russell}, his set contains only  $27 \times 2=54$ inequalities. 

%

Thus, Russell's ``exact core determining class'' is an incomplete description of the IV model. Consequently, a researcher using his set of inequalities may (i) fail to detect observed distributions that violate the IV model, and (ii) fail to provide a sharp bound on the functionals of the joint counterfactual distribution. To illustrate (i), the observed distribution in \cref{tab:Russeldata1} violates the IV model but cannot be detected by Russell's inequalities. For (ii), \cref{tab:sharpnesscompare} compares bounds on several functionals for an observed distribution compatible with the IV model given by \cref{tab:Russeldata2}.

\begin{table}[H]
\caption{An observed distribution that violates the IV model}
\label{tab:Russeldata1}
\resizebox{\textwidth}{!}{
\begin{tabular}{ccccccc}
\hline
 & $P(X=1, Y=1\mid Z)$ & $P(X=1, Y=2\mid Z)$ & $P(X=1, Y=3\mid Z)$ & $P(X=2, Y=1\mid Z)$ & $P(X=2, Y=2\mid Z)$ & $P(X=2, Y=3\mid Z)$ \\ \hline
$Z=1$ & 0.43 & 0.05 & 0.07 & 0.10 & 0.20 & 0.15 \\ \hline
$Z=2$ & 0.01 & 0.36 & 0.40 & 0.18 & 0.03 & 0.02\\ \hline
\end{tabular}}
\end{table}

\begin{table}[H]
\caption{An observed distribution that is compatible with the IV model}
\label{tab:Russeldata2}
\resizebox{\textwidth}{!}{
\begin{tabular}{ccccccc}
\hline
 & $P(X=1, Y=1\mid Z)$ & $P(X=1, Y=2\mid Z)$ & $P(X=1, Y=3\mid Z)$ & $P(X=2, Y=1\mid Z)$ & $P(X=2, Y=2\mid Z)$ & $P(X=2, Y=3\mid Z)$ \\ \hline
$Z=1$ & 0.12 & 0.21 & 0.30 & 0.15 & 0.08 & 0.14 \\ \hline
$Z=2$ & 0.08 & 0.44 & 0.14 & 0.25 & 0.03 & 0.06\\ \hline
\end{tabular}}
\end{table}

\begin{table}[H]
\caption{Bounds on functionals of the counterfactual distribution}
\label{tab:sharpnesscompare}
\resizebox{\textwidth}{!}{
\begin{tabular}{lccccc}
\hline
 & $P(Y(x_1)=2, Y(x_2)=1)$ & $P(Y(x_2)=1)$ & $P(Y(x_1)=1)+P(Y(x_1)=2)$ &  $P(Y(x_1)=1)-P(Y(x_1)=3)$ \\ \hline
Our bound & [0.01, 0.36] & [0.26, 0.78] & [0.56, 0.70] & [-0.32, -0.04]\\ \hline
Russell's bound & [0, 0.36] & [0.17, 0.805] & [0.45, 1.00] & [-0.55, 0.44]\\ \hline
\end{tabular}}
\end{table}

\section{Differences between KM's sets of inequalities and those in this paper}\label{supp:KM}

\citet{KedagniMourifie} (hereafter, KM)
also state inequalities relating observed and potential outcome distributions under various IV models. In this section we compare their results to those in this paper.\footnote{
Note that KM also consider the case where the instrument ($Z$) and possibly the outcome ($Y$) are continuous, which we do not consider here.}

At first glance, the inequalities in \citet{KedagniMourifie}, Supplement p.3,   and those in Theorems \ref{theorem:main-result} and \ref{theo:redundancy} appear similar. Both take the form of upper bounds on probabilities of events involving potential outcomes. 

However, there is a crucial difference in the class of events for which bounds are given, and hence the set of inequalities we consider, versus those considered by K\'{e}dagni and Mourifi\'{e} [KM], Supplement p.3. Specifically, the inequality given there by KM {\bf only} bounds events of the form:
\vskip-28pt
\begin{equation}
P(Y(x\!=\!1) \in A_1,\ldots , Y(x\!=\!K) \in A_K), \label{eq:event}
\end{equation}
where the sets $A_1,\ldots , A_K$ are all members of a {\bf single partition} of the state space ${\cal Y}$ for the outcome $Y$.
In contrast, we allow each of these sets to be arbitrary (non-empty) members of the {\bf power set} for ${\cal Y}$. Thus, every instance of KM's inequality is an instance of ours, but the reverse is not true.\footnote{The use of sets arising from a single partition is integral to the development in Section 2 of KM's supplement; see for example their equation (12), which would become trivially true if instead summation were taken over all non-empty subsets.}
Furthermore, it can be shown that even in the simplest case of binary treatment, the inequalities given by KM are, in general, insufficient to obtain sharp bounds on treatment effects (if $|{\cal Y}|>2$) and insufficient to characterize the set of observable distributions\footnote{
Though KM discuss both treatments and outcomes with more levels, the only setting for which they prove a sharp characterization is when both treatment and outcome are binary.} (if both $|{\cal Y}|, |{\cal Z}|>2$); see \cref{tab:1} below and the subsections for detailed examples.\footnote{In their methodological development KM require the analyst to choose a particular partition, and note that this choice will affect the sharpness of the bounds. However, even if one combines {\bf all} the inequalities arising from KM's inequality considering {\bf all} possible partitions, this is still insufficient.}

\citet{beresteanu} [BMM] apply Artstein's theorem to derive a set of inequalities that includes ours.
However, BMM quantify over {\bf all} subsets of ${\cal Y}^K$.
Consequently, their set contains super-exponentially many more inequalities; see also \citet{chesher:rosen:2017}. This renders the BMM characterization computationally infeasible for practical purposes. It also follows from our Theorem \ref{theo:redundancy} that none of these additional inequalities are required.

\begin{table}
\[
\begin{array}{c c c p{2cm} p{2.5cm} p{2cm} p{2cm}}
\hline
&&&
\multicolumn{2}{c}{\hspace{-1cm}\text{This paper}} 
& \multicolumn{2}{c}{\hspace{-0.6cm}\text{KM}} \\
|{\cal Z}| & |{\cal X}| & |{\cal Y}| &  
\text{Falsification} & \text{ATE Bounds} &
\text{Falsification} & \text{ATE Bounds}
\\
\hline
\text{any} 
& 2 & 2 &  \text{Sharp} 
& \text{Sharp} & \text{Sharp} & \text{Sharp} \\

2 & 2 & 3 & \text{Sharp} 
& \text{Sharp} & \text{Sharp} & \text{Not Sharp} \\

3 & 2 & 3 &  \text{Sharp} 
& \text{Sharp} & \text{Not Sharp} & \text{Not Sharp} \\

4 & 2 & 3 
& \text{Sharp} 
& \text{Sharp} & \text{Not Sharp} & \text{Not Sharp} \\

\hline
\end{array}
\]
\caption{{Results for Binary Treatment $X$ comparing the inequalities given by this paper and those considered by
K\'{e}dagni and Mourifi\'{e} (KM).} For Falsification, `Sharp' implies that the corresponding set of inequalities characterizes the model for the observables, while `Not Sharp' implies that there is an observed distribution incompatible with the model yet for which there exists a potential outcome distribution satisfying all given inequalities. Similarly, for Bounds, `Sharp' implies that every ATE value in the resulting interval may be achieved by some potential outcome distribution compatible with an observed distribution that is in the model, while `Not Sharp' implies that this will not always be the case.
Examples showing the lack of sharpness are given in \cref{app:counterexample:4arm,app:counterexample:3arm,app:counterexample:2arm}.
\label{tab:1} }
\end{table}

\begin{table}
{\small
\resizebox{\linewidth}{!}{
\begin{tabular}{crrrr}
\hline
 & $|{\cal Y}|=2$ & $|{\cal Y}|=3$ & $|{\cal Y}|=4$ & $|{\cal Y}|=5$ \\ \hline
 $|{\cal X}|\!=\!2$ & {\red 4+4}/{8}/{\blue 15} & {\red 12+18}/{42}/{\blue 511} & {\red 28+64}/{204}/{\blue 65,535} & {\red 60+210}/{910}/{\blue 33,554,431}  \\ \hline
 $|{\cal X}|\!=\!3$ & {\red 6+8}/{26}/{\blue 255} & {\red 18+48}/{333}/{\blue 134,217,727} & {\red 42+224}/{3,344}/{\blue $2^{64}\!-\!1$} & {\red 90+960}/29,715/{\blue $2^{125}\!-\!1$}  \\ \hline
 $|{\cal X}|\!=\!4$ & {\red 8+16}/{80}/{\blue 65,535} & {\red 24+126}/{2,388}/{\blue $2^{81}\!-\!1$} & {\red 56+748}/50,584/{\blue $2^{256}\!-\!1$} & {\red 120+3,990}/923,420/{\blue $2^{625}\!-\!1$} \\ \hline
\end{tabular}}}
\caption{{Number of inequalities per instrument arm for different $|{\cal X}|$ and $|{\cal Y}|$}.\\
\textcolor{red}{KM (marginal+partition)}, these are necessary, but not sufficient; our \cref{theo:redundancy}, these are necessary, sufficient and non-redundant; \textcolor{blue}{BMM}, these are necessary and sufficient, but redundant.
KM Marginal Inequalities are those from KM, proof of Proposition 1, equation (56); Partition Inequalities are obtained by considering the inequalities in KM Supplement p.3, but applied to {\bf all} possible partitions of ${\cal Y}$.\label{tab:num-ineqs-km-sgcr-bmm}}
\end{table}

\subsection{Example of inequalities for binary treatment that are not in KM's set}
Consider the simple setting with binary treatment ($|{\cal X}|=2$) and ternary outcome ($|{\cal Y}|=3$). For this setting, \cref{theorem:main-result} yields
\[
(2^3-1)^2-1 = 48
\]
inequalities per instrument arm. \Cref{theo:redundancy} reduces this to $42$ by removing redundant inequalities.

In contrast, even if one considers all possible partitions $P_{\cal Y}$ of the outcome space ${\cal Y}$ and uses the expression given on p.3 of KM's supplement to derive inequalities for each partition, this still leads to only 18 
inequalities on joint potential outcome probabilities. 

For example, although KM do include inequalities such as
\[
P(Y(x\!=\!1) \in \textcolor{blue}{\{1,2\}},Y(x\!=\!2) \in\textcolor{teal}{\{3\}}) \leq
P(Y \in \textcolor{blue}{\{1,2\}}, X\!=\!1 \mid z) +
P(Y \in \textcolor{teal}{\{3\}},  X\!=\!2 \mid z),
\]
for which both sets belong to the single partition
$\{\textcolor{blue}{\{1,2\}},\textcolor{teal}{\{3\}}\}$,  KM do not consider inequalities such as:
\[
P(Y(x\!=\!1) \in \textcolor{blue}{\{1,2\}},Y(x\!=\!2) \in\textcolor{purple}{\{1,3\}}) \leq
P(Y \in \textcolor{blue}{\{1,2\}}, X\!=\!1 \mid  z) +
P(Y \in \textcolor{purple}{\{1,3\}}, X\!=\!2 \mid  z)
\]
since the sets
$
\textcolor{blue}{\{1,2\}},\textcolor{purple}{\{1,3\}}
$
do not belong to a common partition.

In the proof of Proposition 1, equation (56), KM also consider bounds on marginal probabilities of the form $P(Y(x=k) \in A)$. This provides an additional $12$ marginal inequalities, but this still only leads to $12+18=30$ in total.\footnote{Note that marginal inequalities do not correspond to instances of KM's joint inequalities. This is because marginalizing a variable $Y(x_k)$ requires choosing $A_k = {\cal Y}$, the full state space, which in turn implies that the partition is trivial, so that $P_{\cal Y} = \{{\cal Y}\}$.}

Note that a direct application of Artstein's theorem, as considered by \citet{beresteanu} (BMM), would include an inequality for every event of the form:
\[
P\!\left(\,\,\vphantom{2^X}\left(Y(x\!=\!1),Y(x\!=\!2)\right)\in {\mathfrak{A}}\,\,\right)
\]
where ${\mathfrak{A}}$ is an {\bf arbitrary} subset of 
${\cal Y}\times {\cal Y}$. This leads to
\[
2^{{|{\cal Y}\times {\cal Y}|}}-1 = 2^{3^2}-1=2^9-1 = 511
\]
inequalities. It is a direct consequence of \cref{theo:redundancy} that
$511-42=469$ (or $>90\%$) of these inequalities are redundant. 
In general, the number of redundant inequalities arising from Artstein's theorem grows super-exponentially in $M$ and $K$; see \cref{tab:num-ineqs-km-sgcr-bmm}.

\subsection{Analytic example with binary treatment showing KM inequalities are not sharp for falsification with $|{\cal Z}|=4$} \label{app:counterexample:4arm}

It follows directly from \cref{theo:redundancy} and the count of inequalities that the set of inequalities considered by KM is not (in general) sharp for characterizing the identified set of potential outcome distributions, and hence does not (in general) suffice for computing bounds. However, an interested reader might still wonder whether the KM inequalities are sufficient for the purpose of falsification. For example, it is known in the case where the instrument, treatment and outcome are all binary that inequalities on margins alone lead to a sharp test for falsification; bounds on joint potential outcomes are not required.

In this section we answer this question by showing via a simple `analytic' example with binary treatment $X$, ternary outcome $Y$ and four levels of the instrument $Z$ that the set of inequalities considered by KM does not lead to a sharp falsification test.

In this example, the observed distribution given by the first three instrument levels suffices to point identify the joint distribution of potential outcomes.
The observed distribution given by the fourth instrument level violates inequalities in the set given by \cref{theo:redundancy}, but satisfies all inequalities proposed by KM. Consequently, KM would fail to reject the model whereas our characterization would correctly detect incompatibility.





Consider the following observed distribution $P(X,Y\mid Z)$:
\[
\begin{array}{c|@{\quad}ccc@{\quad\quad}ccc}
& \multicolumn{3}{c@{\quad\quad}}{X=1} & \multicolumn{3}{c}{X=2} \\[-6pt]
Z & Y=1 & Y=2 & Y=3 & Y=1 & Y=2 & Y=3\\
\hline
1 & \tfrac13 & \tfrac13 & \tfrac13 & 0 & 0 & 0\\
2 & 0 & 0 & 0 & \tfrac23 & 0 & \tfrac13\\
3 & 0 & \tfrac13 & 0 & \tfrac23 & 0 & 0\\
4 & 0 & \tfrac13 & 0 & \tfrac13 & 0 & \tfrac13
\end{array}
\]
Arms \(Z=1\) and \(Z=2\) correspond to a randomized experiment with perfect compliance. Data from these arms identify the marginals \(P(Y(x\!=\!1))\) and \(P(Y(x\!=\!2))\). Combining these marginals with the distribution observed under \(Z=3\) uniquely identifies the joint distribution to be:
\begin{equation}\label{eq:joint-ided}
P(Y(x\!=\!1),Y(x\!=\!2))
=
\frac13\delta_{(1,1)}
+
\frac13\delta_{(3,1)}
+
\frac13\delta_{(2,3)},
\end{equation}
where $\delta_{(i,j)}$ corresponds to a point mass on the event $\{Y(x\!=\!1)=i,Y(x\!=\!2)=j\}$.
Both our characterization and that of KM agree that the first three arms imply this joint distribution.

Now consider arm \(Z=4\). Under the identified joint distribution (\ref{eq:joint-ided}),
\[
P\!\left(Y(x\!=\!1)\in\{1,3\},\,Y(x\!=\!2)=1\right)
=
\frac23.
\]

However, the observed distribution under \(Z=4\) satisfies
\[
P(Y\in\{1,3\},X=1\mid Z=4)
+
P(Y=1,X=2\mid Z=4)
=
0+\frac13
=
\frac13.
\]

Hence the valid inequality
\begin{align}\label{eq:violation1}
P\!\left(Y(x\!=\!1)\in\{1,3\},\,Y(x\!=\!2)=1\right)
&\le
P(Y\in\{1,3\},X=1\mid Z=4)\nonumber\\
&\quad+
P(Y=1,X=2\mid Z=4).
\end{align}
is violated.
Likewise,
\[
P\!\left(Y(x\!=\! 1)\in\{1,3\},\,Y(x\!=\! 2)\in\{1,2\}\right)
=
\frac23,
\]
while
\[
P(Y\in\{1,3\},X=1\mid Z=4)
+
P(Y\in\{1,2\},X=2\mid Z=4)
=
0+\frac13
=
\frac13,
\]
so that
\begin{align}\label{eq:violation2}
P\!\left(Y(x\!=\!1)\in\{1,3\},\,Y(x\!=\!2)\in\{1,2\}\right)
&\le
P(Y\in\{1,3\},X=1\mid Z=4)\nonumber\\
&\quad +
P(Y\in\{1,2\},X=2\mid Z=4)
\end{align}
is also violated.

These inequalities (\ref{eq:violation1}) and (\ref{eq:violation2}) correspond to the set pairs
\((\{1,3\},\{1\})\) and \((\{1,3\},\{1,2\})\), respectively. Such pairs are excluded from the KM characterization because they do not belong to a common partition. Consequently, the distribution above satisfies all KM inequalities despite violating valid inequalities implied by the model.

Note that in their paper KM present their falsification tests via restrictions on the observed distribution (KM Main paper equations (12), (13), (14), and supplement equations (10), (11), (12)), rather than inequalities relating potential outcome probabilities and observed distributions. However, in their supplement KM prove that all of these inequalities relating observed distributions may be derived from 
bounds on the probability of potential outcome events which are either defined in terms of a single variable, i.e., marginal events, or joint events of the form (\ref{eq:event}), in which all sets are members of a common partition (this is the set of inequalities we refer to as ``KM's inequalities'' in this section). Since the observed distribution in this example is compatible with these inequalities, but is incompatible with the IV model, it also follows that this distribution will not violate any of the restrictions on the observed distribution used by KM in their procedures.

To summarize, if one were to restrict to the inequalities considered by KM, then one would not reject the model and, in this example, would indeed conclude that all potential-outcome functionals are point 
identified. In contrast, our characterization correctly rejects the model. 


\subsection{Numerical example with binary treatment showing KM inequalities are not sharp for falsification with  $|{\cal Z}|=3$}\label{app:counterexample:3arm}

To demonstrate that the phenomenon above is not limited to `degenerate' examples, we also conducted systematic computational searches using linear programming implemented through \texttt{CVXR} \citep{cvxr2020} with $|{\cal Z}|=3, |{\cal X}|=2, |{\cal Y}|=3$.
These searches produced numerous examples in which our characterization rejects the model while the KM characterization remains feasible. 

Consider this observed distribution from a three-arm IV study:
\medskip
\[
\begin{array}{c|@{\quad}ccc@{\quad\quad}ccc}
& \multicolumn{3}{c@{\quad\quad}}{X=1} & \multicolumn{3}{c}{X=2} \\[-6pt]
Z & Y=1 & Y=2 & Y=3 & Y=1 & Y=2 & Y=3\\
\hline
1 & 0.062 & 0.173 & 0.128 & 0.506 & 0.040 & 0.091 \\
2 & 0.049 & 0.342 & 0.056 & 0.260 & 0.135 & 0.158 \\
3 & 0.264 & 0.053 & 0.196 & 0.034 & 0.207 & 0.246
\end{array}
\]

The linear program implied by the KM inequalities remains feasible and returns bounds on $P(Y(x=2)=1)-P(Y(x=1)=1)$ of $[-0.312, -0.094]$. In contrast, the linear program based on our complete characterization is infeasible, thereby proving that the observed distribution is incompatible with the IV model.



\subsection{Numerical example with binary treatment showing KM inequalities are not sharp for bounds; $|{\cal Z}|=2$}\label{app:counterexample:2arm}

We additionally found examples where both characterizations are feasible but KM yields strictly wider bounds for causal parameters. Thus the loss of inequalities has practical consequences not only for falsification but also for identification.
%
\[
\begin{array}{c|@{\quad}ccc@{\quad\quad}ccc}
& \multicolumn{3}{c@{\quad\quad}}{X=1} & \multicolumn{3}{c}{X=2} \\[-6pt]
Z & Y=1 & Y=2 & Y=3 & Y=1 & Y=2 & Y=3\\
\hline
1 & 0.022 & 0.267 & 0.256 & 0.118 & 0.173 & 0.164 \\
2 & 0.425 & 0.132 & 0.041 & 0.068 & 0.133 & 0.201 \\
\end{array}
\]
\noindent The bounds on $P(Y(x=2)=2)-P(Y(x=1)=2)$ obtained under the KM characterization are
\[
[-0.146,\;0.367],
\]
whereas our characterization yields the strictly tighter interval
\[
[-0.065,\;0.234].
\]



\end{document}